\documentclass{article}
\usepackage{amsmath}
\usepackage{amssymb}
\usepackage{mathrsfs}
\usepackage{amsthm}
\usepackage{paralist}
\usepackage{graphicx}
\usepackage{epstopdf}
\usepackage{multirow}
\usepackage[colorlinks=true]{hyperref}
\hypersetup{urlcolor=blue, citecolor=red}
\newtheorem{example}{Example}[section]

\numberwithin{equation}{section}

\newtheorem{lemma}{Lemma}[section]
\newtheorem{remark}{Remark}[section]
\newtheorem{theorem}{Theorem}[section]

\newcommand*\keywords[1]{
\begin{center}{\parbox{14.5cm}{
\noindent\textbf{Keywords \quad}#1\par}}\end{center}
}
\newcommand*\MSC[1]{\begin{center}{
\parbox{14.5cm}{\textbf{\noindent MSC(2010)\quad}#1}}\end{center}
}

\title{An interface-unfitted  finite element method for elliptic interface optimal control problem
\thanks{This work is partially supported by  National Natural Science Foundation of China under grant 11771312, 11471231 and 11401404.}
}
\author{
  Chao Chao Yang \thanks{Email: yangchaochao9055@163.com}, \ Tao Wang  \thanks{Email: 602721197@qq.com }, \
  Xiaoping Xie \thanks{Corresponding author. Email: xpxie@scu.edu.cn} \\
  {School of Mathematics, Sichuan University, Chengdu 610064, China}
}

\date{}
\begin{document}
\maketitle

\begin{abstract}
This paper develops and analyses numerical approximation for linear-quadratic optimal control problem governed by elliptic interface equations. We adopt variational discretization concept to discretize optimal control problem, and apply an interface-unfitted  finite element method due to [A. Hansbo and P. Hansbo. An unfitted finite element method, based on Nitsche's method, for elliptic interface problems. Comput. Methods Appl. Mech. Engrg., 191(47-48): 5537-5552, 2002] to discretize corresponding state and adjoint equations, where  piecewise cut basis functions around interface are enriched into standard conforming finite element space. Optimal error estimates in both $L^2$ norm and a mesh-dependent norm are derived for optimal state, co-state and  control  under different regularity  assumptions. Numerical results verify the theoretical results.
\end{abstract}

\MSC{65K15, 65N30}

\section{Introduction}

Many optimization processes in science and engineering lead to optimal control problems governed by partial differential equations (pdes). In particular in some practical problems, such as the multi-physics progress or engineering design with different materials, the corresponding controlled systems are described by elliptic equations with interface, whose coefficients are discontinuous across the interface.

Let's consider the following  linear-quadratic optimal control problem governed by elliptic interface equations:
\begin{equation}\label{objective-Continue}
       \text{min}~ J(y, u):=\frac{1}{2}\int_\Omega (y-y_d)^2~dx+\frac{\alpha}{2}\int_\Gamma u^2~ds 
\end{equation}
for $ (y,u)\in H^1_0(\Omega)\times L^2 (\Gamma) $ subject to the elliptic interface problem 
\begin{equation}\label{stateEqu-Continue}
\left\{
 \begin{array}{rll}
 & -\nabla\cdot(a(x)\nabla y)=f, & \text{ in }\Omega \\
 & y=0, & \text{ on }\partial\Omega \\
 & [y]=0,[a\nabla_n y ]=g+u, & \text{ on }\Gamma \\
\end{array}
\right.
\end{equation}
with the control constraint
\begin{equation}\label{control-constraint}
u_a\leq u \leq u_b, \text{ a.e. on } \Gamma.
\end{equation}
Here $\Omega\subseteq\mathbb{R}^d (d=2,3)$ is  a polygonal or polyhedral domain, consisting of two disjoint subdomains $\Omega_i (1\leq i\leq 2)$, and interface $\Gamma=\partial\Omega_1\cap \partial\Omega_2$; see Figure \ref{interfaceGeometry-figure} for an illustration. $y_d\in L^2(\Omega)$ is the desired state to be achieved by controlling  $u$ through  interface $\Gamma$, and  $\alpha$ is a positive constant.
$a(\cdot)$ is piecewise constant with
$$a|_{\Omega_i}=a_i>0, i=1,2.$$
$[y]:=(y|_{\Omega_1})|_\Gamma-(y|_{\Omega_2})|_\Gamma$  is  the jump  of function $y$ across interface $\Gamma$, $\nabla_n y=n\cdot\nabla y$  is the normal derivative of $y$ with $n$  denoting the  unit  outward normal vector along $\partial\Omega_1\cap \Gamma$, 
\begin{equation}\label{fguaub}
f\in L^2(\Omega), \quad g\in H^{1/2} (\Gamma), \quad \text{ and } u_a, u_b \in H^{1/2} (\Gamma)\  \text{ with $u_a\leq u_b$ a.e. on $\Gamma$}.
\end{equation}
The choice of homogeneous boundary condition on boundary $\partial\Omega$ is made for ease of presentation, since similar results are valid for other boundary conditions.

\begin{figure}[htbp]
\centering
\includegraphics[width=4cm]{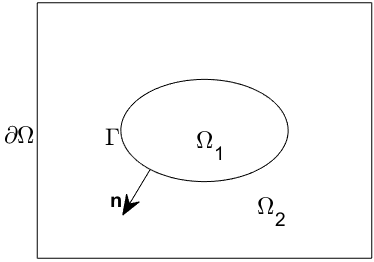}
\caption{The geometry of an interface problem: an illustration} \label{interfaceGeometry-figure}
\end{figure}

For elliptic interface problem, the global regularity of its solution is often low due to the discontinuity of coefficient $a(\cdot)$. The low global regularity may result in reduced accuracy for its finite element approximations \cite{Babuska70fin, Xu82est}, especially when  the interface has complicated geometrical structure \cite{Jeri81, Lius17sec}. Generally there have two categories in literature to tackle this difficulty, i.e.  interface(or body)-fitted methods \cite{Barretts87Fit, Brambles96fin, Chen98, Huangs02mor, Plums03opt, Lis10opt, Xus16opt, Cai17dis, Zhus18spl, Deka18wea} and interface-unfitted methods. For the interface-fitted methods, meshes aligned with the interface are used so as to dominate the approximation error caused by the non-smoothness of solution. In practice, it is usually difficult to construct such meshes, especially in three-dimensional problems.

In contrast, the interface-unfitted methods, with certain types of modifications for approximating functions around interface,  do not require the meshes to fit the interface, and thus avoid complicated mesh generation. For some representative  interface-unfitted methods, we refer to the extended/generalized  finite element method \cite{Melenk95gen, Melenks96par, Moes99fin, Strouboulis00des, Belyschkos09rev}, where additional basis functions characterizing the singularity of solution around interface are enriched into the approximation space, and the immersed finite element method (IFEM) \cite{Li03new, Camps06qua, Li06imm, Fedkiw06imm, Lins07err, Hes12con, Lins15par}, which uses special finite element basis  functions  satisfying the interface jump conditions in a certain sense.

In  \cite{Hans02} an  interface-unfitted finite element method based on Nitsche's approach  \cite{Nitsche71ube} was proposed for elliptic interface equations. In this method,  piecewise linear cut basis functions  around interface are added into the standard linear finite element space,  and  corresponding parameter in the Nitsche's numerical fluxes on each  element intersected by interface are chosen to depend on the relative area/volume of the two parts aside interface. This method was later named as CutFEM in \cite{Hansbo14cut, Burman15cut, Cenanovics16cut, Schott17sta}. In fact, this method can be viewed   as an extended finite element method combined with Nitsche's approach,  which is also called as Nitsche-XFEM \cite{Beckers09nit, Lehrenfelds17opt}.  As shown in  \cite{Hans02}, the CutFEM   yields optimal order convergence, i.e. second order convergence in $L^2$-norm  on a non-degenerate triangulation.

For optimal control problem governed by  elliptic pdes with smooth coefficients $a(\cdot)$ and with  the control $u$ acting in whole domain $\Omega$ or on boundary $\partial\Omega$, a lot of  finite element methods have been studied; see, e.g. \cite{Beckers00ada, Lis02ada, Hinz05var, Benedixs09pos, Lius09new, Hintermuller10goa, Chens10err, Roschs17rel, Kohls14Convergence, Schneiders16pos, Gongs16ada, Wengs16sta, Yangs17rob, Roschs17rel}. However, there are limited literature on the numerical analysis   for  optimal control problems governed by elliptic interface equations.  \cite{Zhang15imm} developed a numerical method, based on the variational discretization concept (cf.\cite{Hinz05var, Hinz09}), for  the case of distributed control, i.e. control $u$ acting in $\Omega$ through
$$-\nabla\cdot(a(x)\nabla y)=f+u,$$
where the IFEM is applied to discretize the state equation with homogeneous interface jump condition
 $$[a\nabla_n y]=0,\quad \text{ on }\Gamma. $$
Optimal error estimates were derived for the control, state and co-state  on uniform triangulations. We note that it is usually difficult to extend the IFEM to the case of non-homogeneous interface conditions \cite{Hes11imm, Hans16a3d, Jis18par}.  \cite{Wachsmuth16opt} investigated $hp$-finite elements for  the model problem \eqref{objective-Continue}-\eqref{control-constraint} on interface-fitted meshes, and didn't give optimal  convergence rates for the state and control in $L^2$ norm.

In this paper,  we'll also adopt the variational discretization concept to discretize the  optimal control problem \eqref{objective-Continue}-\eqref{control-constraint}, and apply the CutFEM   on interface-unfitted meshes for   the state and co-state equations.  Optimal error estimates  in both $L^2$ norm and a mesh-dependent norm will be derived for the optimal state, co-state, and  control   under different regularity  assumptions.

The rest of this paper is organized as follows. In Section 2, we give some notations and optimality conditions for the optimal control problem.  Section 3 sketches  the CutFEM briefly, then complements error estimates of the CutFEM in fractional Sobolev space $H^{3/2}$. In Section 4, we firstly give the discrete optimal control problem and its   optimality conditions, then derives error estimates for the state ,co-state and control of the optimal control problem. Finally, Section 5 provides numerical examples to verify our theoretical results.

\section{Notation and optimality conditions}

For bounded domain $\Lambda \subset  \mathbb{R}^d$ and non-negative integer $m$, let $H^{m}(\Lambda)$ and $H^{m}_0(\Lambda)$  denote the standard    Sobolev spaces on $\Lambda$ with    norm $\|\cdot\|_{m, \Lambda}$ and semi-norm $|\cdot|_{m,\Lambda}$. In particular, $L^2(\Lambda):=H^0(\Lambda)$, with  the standard $L^2$-inner product $(\cdot,\cdot)_\Lambda$. We also need the fractional Sobolev space
$$H^{m+\frac{1}{2}} (\Lambda) := \{w\in H^m (\Lambda): \sum_{|\alpha|=m}\iint_{\Lambda\times\Lambda} \frac{|D^\alpha w(s)-D^\alpha w(t)|^2}{|s-t|^{d+1}}~dsdt <\infty\}$$
with  norm
$$\|w\|_{m+\frac{1}{2},\Lambda} :=\left(\|w\|_{m,\Lambda}^2+\sum_{|\alpha|=m} \iint_{\Lambda\times\Lambda} \frac{|D^\alpha w(s)-D^\alpha w(t)|^2}{|s-t|^{d+1}}~dsdt \right)^{\frac12}.$$
For $s\in \mathbb{R}^+$, let's define
$$ H^s(\Omega_1\cup\Omega_2):=\left\{w\in L^2(\Omega):\ w|_{\Omega_i} \in H^s(\Omega_i),\ i=1,2   \right\}$$
with norm $\|\cdot\|_{s,\Omega_1\cup\Omega_2}:=\left(\sum_{i=1}^2 \|\cdot\|_{s,\Omega_i}^2 \right)^{\frac{1}{2}}$.

The weak formulation of state equation \eqref{stateEqu-Continue} reads: find $y\in H_0^1(\Omega)$ satisfying
\begin{equation}\label{weakFormulation-continue}
a(y,w)=(f,w)_\Omega+(g+u,w)_\Gamma,~~\forall w\in H_0^1(\Omega).
\end{equation}
Where $a(y,w):=( a\nabla y, \nabla w)_\Omega .$

In order to get convergence order of finite element methods, let's make the following  regularity assumptions for above interface equations.
\begin{itemize}
\item[\bf{(R1)}.] If $g+u\in L^2 (\Gamma)$, then the weak solution $y$ of \eqref{weakFormulation-continue}  satisfies $y\in H_0^1 (\Omega)\cap H^{3/2} (\Omega_1 \cup \Omega_2)$ and
\[
\|y\|_{\frac{3}{2},\Omega_1 \cup \Omega_2} \lesssim \|f\|_{L^2 (\Omega)}+\|g+u\|_{L^2 (\Gamma)}.
\]
\item[\bf{(R2)}.] If $g+u\in H^{1/2} (\Gamma)$, then the weak solution $y$ of \eqref{weakFormulation-continue} satisfies $y\in H_0^1 (\Omega)\cap H^2 (\Omega_1 \cup \Omega_2)$ and
\[
\|y\|_{2,\Omega_1 \cup \Omega_2} \lesssim \|f\|_{L^2 (\Omega)}+\|g+u\|_{\frac{1}{2},\Gamma}.
\]
\end{itemize}
Here and in what follows, we use ``$\bar{a}\lesssim \bar{b}$'' to denote that, there is a generic positive constant $C$, independent of the mesh parameter  $h$  and the location of interface relative to the mesh, such that ``$\bar{a}\leq C\bar{b}$.  ``$\bar{a}\approx  \bar{b}$'' means ``$\bar{a}\lesssim \bar{b}\lesssim \bar{a}$''.

\begin{remark} \label{assumption-regularity}

Let's point out that the above assumptions are reasonable. In fact, for the assumption (R1), if   $\Omega$ and   $\Gamma$ are smooth with $\Gamma\cap \partial\Omega=\emptyset$, then (R1) holds \cite[(2.2)]{Brambles96fin}. And it has been shown in  \cite[Corollary 4.12]{Wachsmuth16opt} that (R1) holds if   $\Omega \subset \mathbb{R}^2$ and its subdomains $\Omega_i$ are all polygonal. For the assumption (R2), if the domain $\Omega$ is convex, and the interface $\Gamma$ is $C^2$ continuous with $\Gamma\cap \partial\Omega=\emptyset$, then (R2) also holds \cite[theorem 2.1]{Chen98}.
\end{remark}

For the boxed control constraint
\begin{equation*}
U_{ad}:=\{u\in L^2(\Gamma):u_a\leq u\leq u_b, \text{ a.e. on } \Gamma\},
\end{equation*}
by standard optimality techniques, we can easily derive existence and uniqueness results and optimality conditions for the optimal control problem.
\begin{lemma}
The optimal control problem \eqref{objective-Continue}-\eqref{control-constraint} admits a unique solution  $ (y^*,u^*)\in H^1_0(\Omega)\times  U_{ad}$, and the  equivalent optimality conditions read: find $(y^*,p^*,u^*)\in H_0^1(\Omega)\times H_0^1(\Omega)\times U_{ad}$ such that
\begin{align}
 & a(y^*,w)=(f,w)_\Omega+(g+u^*,w)_\Gamma,~~\forall w\in H_0^1(\Omega),\label{stateEqu2-continue}\\
 & a(w,p^*)=(y^*-y_d,w)_\Omega,~~\forall w\in H_0^1(\Omega),\label{costateEqu-continue}\\
 & (p^*+\alpha u^*,u-u^*)_\Gamma \geq 0,~~\forall u\in U_{ad}.\label{variational-inequality-continue}
\end{align}
\end{lemma}
\begin{proof} For the sake of completeness, we give a brief proof. 
For $u\in L^2 (\Gamma)$, the weak problem \eqref{weakFormulation-continue} admits a unique weak solution $y=y(u)$. Let's introduce  a reduced functional $\tilde J(u):=J(y(u),u)$. Then 
 the existence and uniqueness of $u$ follow from that   $\tilde J(\cdot)$ is strictly convex and continuous in $U_{ad}$.  The equations \eqref{stateEqu2-continue}-\eqref{variational-inequality-continue} are  necessary conditions for the  optimal control problem \eqref{objective-Continue}-\eqref{control-constraint}. And,  from the convexity of $J(\cdot)$, they  are also sufficient conditions (cf. \cite{Lion71, Troltzsch10opt}).
\end{proof}

\begin{remark}
We note that   \eqref{costateEqu-continue} is the so-called adjoint equations, and $p^*$ is the co-state,  which is the weak solution of following interface equations
\[
\left\{
\begin{array}{rll}
 & -\nabla\cdot(a(x)\nabla p^*)=y^*-y_d, & \text{ in }\Omega, \\
 & p^*=0,& \text{ on }\partial\Omega, \\
 & [p^*]=0,[a\nabla_n p^*]=0.& \text{ on }\Gamma.
\end{array}
\right.
\]
\end{remark}
\begin{remark}
The variational inequality \eqref{variational-inequality-continue} means that
\begin{equation}\label{control-projection}
u^*=P_{U_{ad}} (-\frac{1}{\alpha} p^*|_\Gamma),
\end{equation}
where $P_{U_{ad}}$ denotes the $L^2-$projection onto $U_{ad}$ \cite{Brezis11fun}.
\end{remark}

\begin{lemma}
Assume that \eqref{fguaub} holds,  and let $ (y^*,u^*)\in H^1_0(\Omega)\times  U_{ad}$ be the solution to the optimal control problem \eqref{objective-Continue}-\eqref{control-constraint}.  Then, under  the assumption (R2),  we have $u^*\in H^{1/2}(\Gamma), y^*\in H^2(\Omega_1 \cup \Omega_2)$, and
\begin{equation}
\|y^*\|_{2,\Omega_1 \cup \Omega_2}\lesssim \|f\|_{L^2(\Omega)}+\|g\|_{\frac{1}{2},\Gamma}+\|u^*\|_{\frac{1}{2},\Gamma}.
\end{equation}
\end{lemma}
\begin{proof}
From (R2), it suffices to  show  $u^*\in H^{1/2}(\Gamma)$. Since  $p^*\in H^1(\Omega)$,  
from \eqref{control-projection} it follows $u^*=P_{U_{ad}} (-\frac{1}{\alpha} p^*|_{\Gamma})\in U_{ad}$. As the control constraint in $U_{ad}$ is a boxed one with $u_a, u_b\in H^{1/2}(\Gamma)$,  we obtain $u^*\in H^{1/2}(\Gamma)$ (cf. \cite{Troltzsch10opt}).
\end{proof}

\section{ CutFEM for   state and co-state equations}

We know that  the optimal  state $y^*$ and co-state $p^*$ of \eqref{stateEqu2-continue}-\eqref{variational-inequality-continue}  can respectively be viewed as solutions to the following two interface problems.

Find $y^*\in H_0^1(\Omega)$ such that
\begin{equation} \label{stateEqu3-continue}
a(y^*,w)=(f,w)_\Omega+(g+u^*,w)_\Gamma, ~~\forall w\in H_0^1(\Omega).
\end{equation}

Find $p^*\in H_0^1(\Omega)$ such that
\begin{equation} \label{costateEqu3-continue}
 a(w,p^*)=(y^*-y_d,w)_\Omega,~~\forall w\in H_0^1(\Omega).
\end{equation}

\subsection{Cut finite element schemes}

Let $\mathscr{T}_h$ be a shape-regular triangulation of $\Omega$ consisting of open triangles/tetrahedrons, and mesh size $h=\max_{K\in \mathscr{T}_h}h_K$, where $h_K$ denotes the diameter of $K\in \mathscr{T}_h$.  We mention that  $\mathscr{T}_h$ is independent of the location of interface, and elements of $\mathscr{T}_h$ fall into the following three classes:
\begin{align*}
&G_h:=\{K\in \mathscr{T}_h:K\cap \Gamma \neq \emptyset\},\\
&\bar{G}_{i,h}:=\{K\in \mathscr{T}_h:K\notin G_h \text{ and } K\subset \Omega_i\},\quad i=1,2.
\end{align*}
For element $K\in G_h$, which is called as interface element,  let's set $K_i:=K\cap \Omega_i (i=1,2),  \Gamma_K:=\Gamma\cap K$, and denote by $\Gamma_{K,h}$   the straight line/plane connecting the intersection between $\Gamma$ and $\partial K$.

For ease of discussion,   we make the following   assumptions on $\mathscr{T}_h$ and   $\Gamma$ (cf. \cite{Hans02, Schott17sta}).
\begin{itemize}
\item[\bf{(A1)}.] For $K\in G_h$ and    an edge/face $F\subset \partial K$,    $\Gamma\cap F$ is simply connected.
\item[\bf{(A2)}.] For $K\in G_h$, there is a piecewise smooth function $\delta$ which maps  $\Gamma_{K,h}$ to $ \Gamma_K$. 
\end{itemize}
\begin{remark}
Assumptions (A1)-(A2) are easy to satisfy. In $\mathbb{R}^2$, (A1) means that the interface $\Gamma$ intersects each edge of interface element $K\in G_h$ at most once.  And (A2) means that the part of interface $\Gamma$ contained in each interface element  $K\in G_h$ is piecewise smooth.
\end{remark}

Now let's introduce finite dimensional spaces, for $i=1,2$,
\begin{align*}
V_i^h:=\{\phi\in H^1(\Omega_i):\phi|_{K_i}\text{ is a linear polynomial},~\forall K\in G_h\cup\bar{G}_{i,h}, \text{ and }\phi|_{\partial\Omega\cap\partial\Omega_i}=0\},
\end{align*}
\begin{align}
V^h:=\left\{\phi\in L^2(\Omega): \phi|_{\Omega_i}\in V_i^h,\ i=1,2 \right\},
\end{align}
and define two functions $\kappa_1,\kappa_2$ on  $\Gamma$ by
$$\kappa_i|_{\Gamma_K}=\frac{|K_i|}{|K|}, \forall K\in G_h~~(i=1,2),$$
where $|K_i|$ and $|K|$ denote the area/volume of $K_i$ and $K$ respectively.  It is evident that
$$\kappa_1+\kappa_2=1.$$
For $\phi\in V^h$, we set $\phi_i:=\phi |_{\Omega_i}$ for $i=1,2$, and
\begin{equation*}
\{\phi\}:=\kappa_1\cdot\phi_1|_{\Gamma}+\kappa_2\cdot\phi_2 |_{\Gamma}.
\end{equation*}
Then the cut finite element schemes for  \eqref{stateEqu3-continue} and \eqref{costateEqu3-continue} are described respectively as follows:

Find $y^h\in V^h$ such that
\begin{equation}
a_h(y^h,w_h)=(f,w_h)_\Omega+(g+u^*,\kappa_2 w_{h,1}+\kappa_1 w_{h,2})_\Gamma,\quad \forall w_h\in V^h. \label{cut-stateEqu}
\end{equation}

Find $p^h \in V^h$ such that
\begin{equation}
a_h(p^h,w_h)=(y^*-y_d,w_h)_\Omega, \quad \forall w_h\in V^h. \label{cut-costateEqu}
\end{equation}
The modified bilinear form $a_h(\cdot,\cdot)$ is given by
\begin{align*}
a_h(y^h,w_h):=& \sum_{i=1}^2  (a\nabla y^h, \nabla w_h)_{\Omega_i} \\
& -([y^h],\{a\nabla_n w_h\})_\Gamma-(\{a\nabla_n y^h\},[w_h])_\Gamma +\lambda ([y^h],[w_h])_\Gamma,
\end{align*}
and the stabilization parameter $\lambda$ is taken as
\begin{equation}\label{stablization-constant}
\lambda|_K= \widetilde{C} h_K^{-1} \max\{a_1,a_2\},
\end{equation}
with the constant $\widetilde{C}>0$ sufficiently large.

Let's introduce a mesh-dependent semi-norm $|||\cdot|||$ in $H^{3/2}(\Omega_1\cup\Omega_2)$
with
\begin{equation*}
|||w|||^2:=|w|_{1,\Omega_1\cup\Omega_2}^2+\|[w]\|_{\frac{1}{2},h,\Gamma}^2+\|\{\nabla_n w\}\|_{-\frac{1}{2},h,\Gamma}^2, \quad \forall w\in H^{3/2}(\Omega_1\cup\Omega_2),
\end{equation*}
where \[
 \|\cdot\|_{\frac{1}{2},h,\Gamma}^2:=\sum_{K\in G_h} h_k^{-1}\|\cdot\|_{0,\Gamma_K}^2,~~
 \|\cdot\|_{-\frac{1}{2},h,\Gamma}^2:=\sum_{K\in G_h} h_K\|\cdot\|_{0,\Gamma_K}^2.
\]
It is easy to see that $|||\cdot|||$ is a norm on $V^h$ and it holds
\begin{equation}\label{poincare-inequality}
||w_h||_{0,\Omega}\lesssim |w_h|_{1,\Omega_1\cup\Omega_2} \leq |||w_h|||, \quad \forall w_h\in V^h.
\end{equation}

Then we have the following  boundedness and coerciveness for the bilinear form $a_h(\cdot,\cdot)$ (cf. \cite[Lemma 5]{Hans02}):
\begin{lemma}
It holds
\begin{equation} \label{boundness-discreteBilinear}
a_h(y,w)\lesssim |||y|||\ |||w|||, \quad \forall y, w\in H^{3/2}(\Omega_1\cup\Omega_2).
\end{equation}
In addition, if $\widetilde{C}$ of \eqref{stablization-constant} is chosen to be sufficiently large, then
\begin{equation}\label{coercivity-discreteBilinear}
a_h(w_h,w_h)\gtrsim |||w_h|||^2, \quad \forall w_h\in V^h.
\end{equation}
\end{lemma}

\begin{remark}
As shown in \cite[lemma 1]{Hans02},   the schemes  \eqref{cut-stateEqu}-\eqref{cut-costateEqu} are consistent with respect to the weak solutions $y^*,p^* \in H^1_0(\Omega)$ of problems \eqref{stateEqu3-continue}-\eqref{costateEqu3-continue} respectively in the following sense: for $w_h\in V^h$, we have
\begin{equation}\label{cut-consistent}
a_h(y^*-y^h,w_h)=0, \quad a_h(p^*-p^h,w_h)=0.
\end{equation}
\end{remark}

From \cite{Hans02}, the following results of existence, uniqueness, and error estimates hold:
\begin{lemma}
Assume $g,u^*\in H^{1/2} (\Gamma) $, and $y^*,p^* \in H^1_0(\Omega)\cap H^{2}(\Omega_1 \cup \Omega_2)$ be the solutions to continuous problems \eqref{stateEqu3-continue}-\eqref{costateEqu3-continue} respectively. If $\widetilde{C}$ of \eqref{stablization-constant} is chosen to be sufficiently large, then
\noindent
(i) The discrete scheme \eqref{cut-stateEqu} admits a unique solution $y^h\in V^h$ such that
\begin{align}
|||y^*-y^h|||\lesssim h \|y^*\|_{2,\Omega_1\cup\Omega_2},\label{errorMeshnorm-cutH2regularity-state} \\
\|y^*-y^h\|_{0,\Omega}\lesssim h^2\|y^*\|_{2,\Omega_1\cup\Omega_2}. \label{errorL2norm-cutH2regularity-state}
\end{align}
(ii) The discrete scheme \eqref{cut-costateEqu} admits a unique solution  $p^h\in V^h$ such that
\begin{align}
|||p^*-p^h|||\lesssim h \|p^*\|_{2,\Omega_1\cup\Omega_2},\label{errorMeshnorm-cutH2regularity-costate} \\
\|p^*-p^h\|_{0,\Omega}\lesssim h^2\|p^*\|_{2,\Omega_1\cup\Omega_2}. \label{errorL2norm-cutH2regularity-costate}
\end{align}
\end{lemma}

\begin{remark}
We note that the error estimates in above lemma require that $y^*,p^*\in H^2( \Omega_1\cup\Omega_2)$. For $y^*$, this means that $g+u^*\in H^{1/2} (\Gamma)$ (cf. \cite{Hans02} and the assumption (R2)).  In next section, we'll derive estimates under  mild regularity assumptions, say  $y^*,p^*\in H^{3/2}( \Omega_1\cup\Omega_2)$.
\end{remark}

\subsection{ Alternative error estimates of  CutFEM}

For $i= 1,2$, let $E_i: \left\{ w\in H^\frac{3}{2}(\Omega_i):\ w|_{\partial \Omega_i\setminus\Gamma}=0\right\} \longrightarrow H^\frac{3}{2}(\Omega)\cap H^1_0 (\Omega)$  be  the extension operators satisfying that, for $w\in H^\frac{3}{2}(\Omega_1\cup\Omega_2)\cap H^1_0 (\Omega)$ with $w_i:=w|_{\Omega_i}$, we have
 \begin{equation}\label{extension-pisewise}
E_iw_i|_{\Omega_i}=w_i,\quad \|E_i w_i\|_{s,\Omega}\lesssim \|w_i\|_{s,\Omega_i},~0\leq s\leq \frac{3}{2}.
 \end{equation}

Let $I_h: H^1_0 (\Omega)\longrightarrow \{w\in C(\bar\Omega): w|_K \text{ is linear },\forall K\in \mathscr{T}_h, \text{ and } w|_{\partial\Omega}=0\}$ denote the Scott-Zhang interpoation operator \cite{Scot90}. Then for $K\in \mathscr{T}_h$ and $ m=1,2$, we have
\[  \|w-I_h w\|_{j,K}\lesssim h_K^{m-j} \|w\|_{m,S_K},\quad \forall w\in H^m (\Omega)\cap H^1_0 (\Omega) , \ j=0,1\]
where $S_K:=\text{interior}\{\cup \bar{T}: T\in \mathscr{T}_h, \bar{T}\cap \bar{K}\neq \emptyset\}$.
 Thus, by using the real interpolation method (cf. the proof of \cite[Theorem (14.3.3)]{Bren08}), it's easy to get estimation
 \begin{equation}\label{interpolationSZ-error}
\|w-I_h w\|_{j,K}\lesssim h_K^{\frac{3}{2}-j}\|w\|_{\frac{3}{2},S_K}, \forall  w\in H^\frac{3}{2} (\Omega)\cap H^1_0 (\Omega), \ K\in \mathscr{T}_h, \ j=0,1.
\end{equation}

Now we construct an interpolation operator $I_h^*:H^1_0 (\Omega)\cap H^\frac{3}{2}(\Omega_1\cup\Omega_2)\longrightarrow V^h$ with
 \begin{equation}\label{interpolation-operator}
 (I_h^* w)|_{\Omega_i}:=(I_h E_i w_i)|_{\Omega_i}, \quad  i=1,2.
 \end{equation}

\begin{lemma}
For $w\in H^1_0 (\Omega)\cap H^\frac{3}{2}(\Omega_1\cup\Omega_2)$, we have
\begin{equation} \label{errorEst-interpolation}
|||w-I_h^* w||| \lesssim h^{\frac{1}{2}} \|w\|_{\frac{3}{2},\Omega_1\cup\Omega_2}.
\end{equation}
\end{lemma}
\begin{proof}
For $w\in H^1_0 (\Omega)\cap H^\frac{3}{2}(\Omega_1\cup\Omega_2)$ with $w_i:=w|_{\Omega_i}$ $(i=1,2)$, from \eqref{extension-pisewise}-\eqref{interpolation-operator} it follows
\begin{align*}
|w-I_h^* w|_{1,\Omega_1\cup\Omega_2}^2 &= \sum_{i=1}^2 \sum_{K\in \mathscr{T}_h} |w_i-I_{h} E_iw_i|_{1,K\cap\Omega_i}^2
= \sum_{i=1}^2 \sum_{K\in \mathscr{T}_h} |E_iw_i-I_{h} E_iw_i|_{1,K\cap\Omega_i}^2   \\
 &\lesssim \sum_{i=1}^2 h\|E_iw_i\|_{\frac{3}{2},\Omega}^2\lesssim \sum_{i=1}^2 h\|w_i\|_{\frac{3}{2},\Omega_i}^2\lesssim h\|w\|_{\frac{3}{2},\Omega_1\cup\Omega_2}^2.
\end{align*}
In light of \eqref{extension-pisewise}-\eqref{interpolation-operator}  and the  trace inequality, we have

\begin{align*}
\|[w-I_h^* w]\|_{\frac{1}{2},h,\Gamma}^2 &\lesssim \sum_{i=1}^2 \sum_{K\in G_h} h_K^{-1}\|w_i-I_{h}E_i w_i\|_{0,\Gamma_K}^2=\sum_{i=1}^2 \sum_{K\in G_h} h_K^{-1}\|E_iw_i-I_{h}E_i w_i\|_{0,\Gamma_K}^2 \\
 &\lesssim \sum_{i=1}^2 \sum_{K\in G_h} (h_K^{-2}\|E_iw_i-I_{h}E_i w_i\|_{0,K}^2+\|E_iw_i-I_{h}E_i w_i\|_{1,K}^2)\\
 &\lesssim \sum_{i=1}^2 \sum_{K\in G_h} h_K\|E_iw_i\|_{\frac{3}{2},S_K}^2 \\
 &\lesssim \sum_{i=1}^2 h\|E_iw_i\|_{\frac{3}{2},\Omega}^2 \lesssim \sum_{i=1}^2 h\|w_i\|_{\frac{3}{2},\Omega_i}^2\lesssim h\|w\|_{\frac{3}{2},\Omega_1\cup\Omega_2}^2.
\end{align*}
Similarly, we obtain
\begin{align*}
\|\{\nabla_n (w-I_h^* w)\}\|_{-\frac{1}{2},h,\Gamma}^2 &\lesssim \sum_{i=1}^2 \sum_{K\in G_h} h_K\|\nabla_n(w_i-I_{h} E_i w_i)\|_{0,\Gamma_K}^2 \\
&\lesssim \sum_{i=1}^2 \sum_{K\in G_h} h_K \|E_iw_i\|_{\frac{3}{2},K}^2 \\
&\lesssim \sum_{i=1}^2 h \|E_iw_i\|_{\frac{3}{2},\Omega}^2 \lesssim \sum_{i=1}^2 h \|w_i\|_{\frac{3}{2},\Omega_i}^2 \lesssim h \|w\|_{\frac{3}{2},\Omega_1\cup\Omega_2}^2.
\end{align*}
Together with  above three estimations we yield the desired conclusion.
\end{proof}

In view of the above lemma, we can  obtain  the following   error estimates for the cut finite element schemes \eqref{cut-stateEqu}-\eqref{cut-costateEqu} under  milder regularity requirement.
\begin{theorem}
Under the assumption (R1), let $y^*,p^* \in H^1_0(\Omega)\cap H^{3/2}(\Omega_1 \cup \Omega_2)$ be the solutions to the continuous problems \eqref{stateEqu3-continue}-\eqref{costateEqu3-continue} respectively, and  let  $y^h, p^h\in V^h$ be the solutions to  the discrete schemes \eqref{cut-stateEqu}-\eqref{cut-costateEqu} respectively. Then we have
\begin{align}
|||y^*-y^h|||\lesssim h^\frac{1}{2} \|y^*\|_{\frac{3}{2},\Omega_1\cup\Omega_2},\label{error-meshnorm-state-lowRegularity}\\
|||p^*-p^h|||\lesssim h^\frac{1}{2} \|p^*\|_{\frac{3}{2},\Omega_1\cup\Omega_2}\label{error-meshnorm-costate-lowRegularity},\\
\|y^*-y^h\|_{0,\Omega} \lesssim h \|y^*\|_{\frac{3}{2},\Omega_1\cup\Omega_2},\label{error-l2norm-state-lowRegularity}\\
\|p^*-p^h\|_{0,\Omega} \lesssim h \|p^*\|_{\frac{3}{2},\Omega_1\cup\Omega_2}\label{error-l2norm-costate-lowRegularity}.
\end{align}
\end{theorem}

\begin{proof}
The estimates \eqref{error-meshnorm-state-lowRegularity}-\eqref{error-meshnorm-costate-lowRegularity} follow from \eqref{errorEst-interpolation}, the discrete coerciveness \eqref{coercivity-discreteBilinear}, and the triangle inequality directly. It remains to show \eqref{error-l2norm-state-lowRegularity}) by using the Nitsche's technique, since \eqref{error-l2norm-costate-lowRegularity} follows similarly.

Consider the interface problem
\begin{align*}
\left\{
\begin{array}{rll}
 & -\nabla\cdot(a(x)\nabla z)=y^*-y^h& \text{in }\Omega, \\
 & z=0& \text{on }\partial\Omega, \\
 & [z]=0,[a\nabla_n z]=0& \text{on }\Gamma.
\end{array}
\right.
\end{align*}
whose equivalent weak problem reads:  find $z\in H_0^1(\Omega)$ satisfying
\begin{equation} \label{proof1-errorLowRegularity}
 a(z,w)=(y^*-y^h,w)_\Omega,~~\forall w\in H_0^1(\Omega).
\end{equation}
Then by the assumption (R1), we have $z\in H^\frac{3}{2}(\Omega_1\cup \Omega_2)$ and
\begin{equation*}
\|z\|_{\frac{3}{2},\Omega_1\cup \Omega_2}\lesssim \|y^*-y^h\|_{0,\Omega}.
\end{equation*}
Let $z_h\in V^h$ denote the CutFEM approximation of $z$, which means that
\begin{equation} \label{proof2-errorLowRegularity}
z_h\in V^h: a_h (z_h,w_h)=(y^*-y^h,w_h)_\Omega,\forall w_h\in V^h.
\end{equation}
Similar with \eqref{error-meshnorm-state-lowRegularity}, we derive that
\begin{align}
|||z-z_h||| &\lesssim h^\frac{1}{2} \|z\|_{\frac{3}{2},\Omega_1\cup\Omega_2} \notag  \\
            &\lesssim h^\frac{1}{2} \|y^*-y^h\|_{0,\Omega} \label{proof3-errorLowRegularity}.
\end{align}
In \eqref{proof1-errorLowRegularity} with $y^*\in H_0^1 (\Omega)$ as the test function we have
\begin{equation}\label{proof4-errorLowRegularity}
a(z,y^*)=(y^*-y^h,y^*)_\Omega.
\end{equation}
With the consistency \eqref{cut-consistent} we have
\begin{equation}\label{proof5-errorLowRegularity}
a_h(z,y^h)=a_h(z_h,y^h)=a_h(z_h,y^*),
\end{equation}
Together with \eqref{proof2-errorLowRegularity}, \eqref{proof4-errorLowRegularity}, the interface conditions $[z]|_\Gamma=0$, and the boundedness \eqref{boundness-discreteBilinear}, we have
\begin{align*}
& \|y^*-y^h\|_{0,\Omega}^2 =(y^*-y^h,y^*)_\Omega-(y^*-y^h,y^h)_\Omega\\
& = a(z,y^*)-a_h(z,y^h) \\
& = a_h(z-z_h,y^*-y^h)\\
& \lesssim |||z-z_h|||~|||y^*-y^h|||.
\end{align*}
In addition with \eqref{error-meshnorm-state-lowRegularity} and \eqref{proof3-errorLowRegularity}, we have
\begin{align*}
\|y^*-y^h\|_{0,\Omega}^2
& \lesssim h^\frac{1}{2} h^\frac{1}{2} \|z\|_{\frac{3}{2},\Omega_1\cup \Omega_2} \|y^*\|_{\frac{3}{2},\Omega_1\cup \Omega_2} \\
& \lesssim h \|y^*-y^h\|_{0,\Omega} \|y^*\|_{\frac{3}{2},\Omega_1\cup \Omega_2},
\end{align*}
which implies the desired result \eqref{error-l2norm-state-lowRegularity}. This completes the proof.
\end{proof}
\begin{remark}
Notice that estimations \eqref{error-meshnorm-state-lowRegularity}-\eqref{error-meshnorm-costate-lowRegularity} are optimal, which indicate
$$|y^*-y^h|_{1,\Omega_1\cup\Omega_2}\lesssim h^\frac{1}{2} \|y^*\|_{\frac{3}{2},\Omega_1\cup\Omega_2},\quad |p^*-p^h|_{1,\Omega_1\cup\Omega_2}\lesssim h^\frac{1}{2} \|p^*\|_{\frac{3}{2},\Omega_1\cup\Omega_2}.$$
\end{remark}

In what follows, we'll show that the convex combination $\kappa_2 y^h_1+\kappa_1 y^h_2$ and $\kappa_2 p^h_1+\kappa_1 p^h_2$ are  ``good'' approximations to $y^*$ and $p^*$  on $\Gamma$ respectively (recall that $y^h_i:=y^h|_{\Omega_i}, p^h_i:=p^h|_{\Omega_i}$, for $i=1,2$).

\begin{theorem}
Let $y^*,p^* \in H^1_0(\Omega)\cap H^s (\Omega_1 \cup \Omega_2)$ ($s=3/2, 2$) be the solutions of continuous problems \eqref{stateEqu3-continue}-\eqref{costateEqu3-continue} respectively, and $y^h, p^h\in V^h$ be the solutions of discrete schemes \eqref{cut-stateEqu}-\eqref{cut-costateEqu} respectively. Then for $s=3/2, 2$, we have
\begin{align}
&\|y^*-(\kappa_2 y^h_1+\kappa_1 y^h_2)\|_{0,\Gamma}\lesssim h^{s-\frac{1}{2}} \|y^*\|_{s,\Omega_1\cup\Omega_2},\label{errorInterface-state-lowRegularity} \\
&\|p^*-(\kappa_2 p^h_1+\kappa_1 p^h_2)\|_{0,\Gamma}\lesssim h^{s-\frac{1}{2}} \|p^*\|_{s,\Omega_1\cup\Omega_2}.\label{errorInterface-costate-lowRegularity}
\end{align}
\end{theorem}

\begin{proof} It suffices to show  \eqref{errorInterface-state-lowRegularity}, since \eqref{errorInterface-costate-lowRegularity} follows similarly. We'll also use Nitsche's technique. Let $z$ be the weak solution of following  interface problem
\begin{align*}
\left\{
\begin{array}{rll}
 & -\nabla\cdot(a(x)\nabla z)=0,& \text{in }\Omega, \\
 & z=0,& \text{on }\partial\Omega, \\
 & [z]=0,[a\nabla_n z]=y^*-(\kappa_2 y^h_1+\kappa_1 y^h_2).& \text{on }\Gamma. \\
\end{array}
\right.
\end{align*}
Whose weak formulation reads: find $z\in H_0^1(\Omega)$ satisfying
\begin{equation}\label{proof1-errorInterface-LowRegularity}
a(z,w)=\left(y^*-(\kappa_2 y^h_1+\kappa_1 y^h_2),w\right)_\Gamma,~~\forall w\in H_0^1(\Omega).
\end{equation}
Since $y^*-(\kappa_2 y^h_1+\kappa_1 y^h_2)\in L^2(\Gamma)$, we get $z\in H_0^1(\Omega)\cap H^\frac{3}{2}(\Omega_1\cup \Omega_2)$ and
\begin{equation*}
\|z\|_{\frac{3}{2},\Omega_1\cup \Omega_2}\lesssim \|y^*-(\kappa_2 y^h_1+\kappa_1 y^h_2)\|_{0, \Gamma}.
\end{equation*}
Let $z_h\in V^h$ denote the CutFEM approximation of $z$, which means $z_h$ satisfies
\begin{equation}\label{proof2-errorInterface-LowRegularity}
a_h(z_h,w_h)=(y^*-(\kappa_2 y^h_1+\kappa_1 y^h_2),\kappa_2 w_{h,1}+\kappa_2 w_{h,2})_\Gamma,~~\forall w_h\in V^h.
\end{equation}
Similar with \eqref{error-meshnorm-state-lowRegularity}, we have
\begin{align}
|||z-z_h||| &\lesssim h^\frac{1}{2} \|z\|_{\frac{3}{2},\Omega_1\cup\Omega_2}\notag \\
            &\lesssim h^\frac{1}{2} \|y^*-(\kappa_2 y^h_1+\kappa_1 y^h_2)\|_{0,\Gamma}\label{proof3-errorInterface-LowRegularity}.
\end{align}
In \eqref{proof2-errorInterface-LowRegularity} with $y^*\in H_0^1 (\Omega)$ as the test function we have
\begin{equation}\label{proof4-errorInterface-LowRegularity}
a(z,y^*)=(y^*-(\kappa_2 y^h_1+\kappa_1 y^h_2),y^*)_\Gamma.
\end{equation}
By the consistency \eqref{cut-consistent} we have
\begin{equation}\label{proof5-errorInterface-LowRegularity}
a_h(z,y^h)=a_h(z_h,y^h)=a_h(z_h,y^*),
\end{equation}
which,  together with  \eqref{proof2-errorInterface-LowRegularity}, \eqref{proof4-errorInterface-LowRegularity},  the interface conditions $[z]|_\Gamma=0$, and the boundedness \eqref{boundness-discreteBilinear}, indicates
\begin{align}
 \|y^*-(\kappa_2 y^h_1+\kappa_1 y^h_2)\|_{0,\Gamma}^2 &=(y^*-(\kappa_2 y^h_1+\kappa_1 y^h_2),y^*)_\Gamma-(y^*-(\kappa_2 y^h_1+\kappa_1 y^h_2),\kappa_2 y^h_1+\kappa_1 y^h_2)_\Gamma \nonumber\\
& = a(z,y^*)-a_h(z,y^h) \nonumber\\
&=a_h(z-z_h,y^*-y^h)\nonumber\\
& \lesssim |||z-z_h|||~|||y^*-y^h|||.\label{proof6-errorInterface-LowRegularity}
\end{align}
This inequality, together with \eqref{errorL2norm-cutH2regularity-state}, \eqref{error-l2norm-state-lowRegularity} and the estimation \eqref{proof3-errorInterface-LowRegularity}, yields
\begin{align*}
\|y^*-(\kappa_2 y^h_1+\kappa_1 y^h_2)\|_{0,\Gamma}
& \lesssim h^{\frac{1}{2}}h^{s-1} \|y^*\|_{s,\Omega_1\cup \Omega_2}\\
& \lesssim h^{s-\frac{1}{2}} \|y^*\|_{s,\Omega_1\cup \Omega_2},\quad (s=3/2, 2).
\end{align*}
This completes the proof.
\end{proof}

\section{Discrete optimal control problem}
\subsection{Discrete optimality conditions}

With variational discretization concept (cf. \cite{Hinz05var, Hinz09}), the optimal control problem \eqref{objective-Continue}-\eqref{control-constraint} is approximated by the following  discrete optimal control problem
\begin{equation}\label{objective-discrete}
\min_{y_h\in V^h, u\in U_{ad}} J_h(y_h, u)=\frac{1}{2}\|y_h-y_d\|_{0, \Omega}^2+\frac{\alpha}{2}\|u\|_{0,\Gamma}^2,
\end{equation}
where $y_h=y_h (u)$ satisfies 
\begin{equation}\label{stateEqu-discrete}
a_h(y_h,w_h)=(f,w_h)_\Omega+(g+u,\kappa_2 w_{h,1}+\kappa_1 w_{h,2})_\Gamma,~\forall w_h\in V^h.
\end{equation}

Similar to the continuous case, it holds the following   existence and uniqueness result and optimality conditions.
\begin{lemma}
The discrete optimal control problems \eqref{objective-discrete}-\eqref{stateEqu-discrete} admits a unique solution $(y_h^*,u_h^*)\in V^h\times U_{ad}$, and its equivalent optimality conditions read: find $(y_h^*,p_h^*,u_h^*)\in V^h\times V^h\times U_{ad}$ such that
\begin{align}
 & a_h(y_h^*,w_h)=(f,w_h)_\Omega+(g+u_h^*,\kappa_2 w_{h,1}+\kappa_1 w_{h,2})_\Gamma,\forall w_h\in V^h, \label{stateEqu2-discrete}\\
 & a_h(w_h,p_h^*)=(y_h^*-y_d,w_h)_\Omega,\forall w_h\in V^h, \label{costateEqu-discrete}\\
 &   (\kappa_2 p_{h,1}^*+\kappa_1 p_{h,2}^*+\alpha u_h^*,u-u_h^*)_\Gamma \geq 0,\forall u\in U_{ad}.\label{variational-inequality-discrete}
\end{align}
\end{lemma}
\begin{remark}
Actually the discrete optimal control $u_h^*\in U_{ad}$   is not directly discretized in the objective functional \eqref{objective-discrete}, since $U_{ad}$ is infinite dimensional. In fact,  the variational inequality \eqref{variational-inequality-discrete} implies that   $u_h^*$ is implicitly discretized through the discrete co-state $p_h^*$ and   the projection $P_{U_{ad}}$  (cf. \eqref{control-projection}) with
\begin{equation*}
u_h^*=P_{U_{ad}}\left(-\frac{\kappa_2 p_{h,1}^*|_\Gamma+\kappa_1 p_{h,2}^*|_\Gamma}{\alpha}\right),
\end{equation*}
as is one  main feature of the variational discretitization concept.
\end{remark}

\subsection{Error estimates}

Firstly let's show that, the errors in $L^2$-norm  or  $|||\cdot|||$-norm between $(y^*,p^*,u^*)$ and $(y_h^*,p_h^*,u_h^*)$, which are the solutions of continuous  optimal control problem  \eqref{stateEqu2-continue}-\eqref{variational-inequality-continue} and  discrete optimal control problem \eqref{stateEqu2-discrete}-\eqref{variational-inequality-discrete} respectively,  is bounded from above by the errors  between $(y^*,p^*)$ and $(y^h,p^h)$, which are the solutions of  \eqref{stateEqu2-continue}-\eqref{costateEqu-continue}  and discrete schemes  \eqref{cut-stateEqu}-\eqref{cut-costateEqu} respectively.
\begin{theorem}
Let $(y^*,p^*,u^*)\in H_0^1(\Omega)\times H_0^1(\Omega)\times U_{ad}$ and $(y_h^*,p_h^*,u_h^*)\in V^h\times V^h\times U_{ad}$ be  the solutions of continuous problem \eqref{stateEqu2-continue}-\eqref{variational-inequality-continue} and discrete problem \eqref{stateEqu2-discrete}-\eqref{variational-inequality-discrete} respectively. Then we have
\begin{align}
\sqrt{\alpha} \|u^*-u_h^*\|_{0,\Gamma}+\|y^*-y_h^*\|_{0,\Omega}\leq & \sqrt{2}\|y^*-y^h\|_{0,\Omega}+\frac{\sqrt{2}}{\sqrt{\alpha}}\|p^*-(\kappa_2 p^h_1+\kappa_1 p^h_2)\|_{0,\Gamma} \label{0error-l2-state}\\
\|p^*-p_h^*\|_{0,\Omega} \lesssim &\|p^*-p^h\|_{0,\Omega} + \|y^*-y_h^*\|_{0,\Omega} \label{0error-l2-costate}\\
|||y^*-y_h^*||| \lesssim &|||y^*-y^h|||+\|u^*-u_h^*\|_{0,\Gamma} \label{0error-meshnorm-state} \\
|||p^*-p_h||| \lesssim &|||p^*-p^h|||+\|y^*-y_h^*\|_{0,\Omega}.\label{0error-meshnorm-costate}
\end{align}
where $y^h, p^h\in V^h$ are the cut finite element solutions of discrete schemes \eqref{cut-stateEqu}-\eqref{cut-costateEqu} respectively.
\end{theorem}
\begin{proof}
We firstly show \eqref{0error-l2-state}. By\eqref{cut-stateEqu}-\eqref{cut-costateEqu} and \eqref{stateEqu2-discrete}-\eqref{costateEqu-discrete} we get
\begin{align}
 & a_h(y_h^*-y^h,w_h)= (u_h^*-u^*,\kappa_2 w_{h,1}+\kappa_1 w_{h,2})_\Gamma,\forall w_h\in V^h, \label{err-statd0}\\
 & a_h(w_h,p_h^*-p^h)=(y_h^*-y^*,w_h)_\Omega,\forall w_h\in V^h, \label{err-adjod0}
\end{align}
which yield
\begin{align}
  (u_h^*-u^*,\kappa_2 (p_{h,1}^*-p^h_1)+\kappa_1 (p_{h,2}^*-p^h_2))_\Gamma= (y_h^*-y^*,y_h^*-y^h)_\Omega. \label{err-statd}
\end{align}
From \eqref{variational-inequality-continue}  and \eqref{variational-inequality-discrete} it follows
\begin{align*}
(p^*+\alpha u^*,u_h^*-u^*)_\Gamma\geq 0,\\
(\kappa_2 p_{h,1}^*+\kappa_1 p_{h,2}^*+\alpha u_h^*,u^*-u_h^*)_\Gamma \geq 0.
\end{align*}
Adding the above two inequalities and using  \eqref{err-statd}, we obtain
\begin{align*}
& \alpha \|u^*-u_h^*\|_{0,\Gamma}^2\leq (\kappa_2 p_{h,1}^*+\kappa_1 p_{h,2}^*-p^*,u^*-u_h^*)_\Gamma \\
 & =(\kappa_2 (p_{h,1}^*-p^h_1)+\kappa_1 (p_{h,2}^*-p^h_2),u^*-u_h^*)_\Gamma+(\kappa_2 p^h_1+\kappa_1 p^h_2-p^*,u^*-u_h^*)_\Gamma\\
 & =-(y_h^*-y^*,y_h^*-y^h)_\Omega+(\kappa_2 p^h_1+\kappa_1 p^h_2-p^*,u^*-u_h^*)_\Gamma\\
 & =-\frac{1}{2}\|y^*-y_h^*\|_{0,\Omega}^2+\frac{1}{2}\|y^*-y^h\|_{0,\Omega}^2+(\kappa_2 p^h_1+\kappa_1 p^h_2-p^*,u^*-u_h^*)_\Gamma\\
 & \leq -\frac{1}{2}\|y^*-y_h^*\|_{0,\Omega}^2+\frac{1}{2}\|y^*-y^h\|_{0,\Omega}^2+\frac{1}{2\alpha}\|p^*-(\kappa_2 p^h_1+\kappa_1 p^h_2)\|_{0,\Gamma}^2+\frac{\alpha}{2}\|u^*-u_h^*\|_{0,\Gamma}^2,
\end{align*}
which implies the  desired conclusion \eqref{0error-l2-state}.

Secondly,  let us   show \eqref{0error-l2-costate}.
 From \eqref{poincare-inequality},   \eqref{coercivity-discreteBilinear}, and \eqref{err-adjod0}, we have
\begin{align*}
|| p_h^*-p^h||^2_{0,\Omega}&\lesssim |||p_h^*-p^h|||^2\\
 &\lesssim a_h(p_h^*-p^h,p_h^*-p^h)=(y_h^*-y^*, p_h^*-p^h)_\Omega\\
 &\lesssim ||y_h^*-y^*||_{0,\Omega}|| p_h^*-p^h||_{0,\Omega},
\end{align*}
which, together with the triangle inequality, leads to the  estimate \eqref{0error-l2-costate}.

Thirdly, let us show \eqref{0error-meshnorm-state}.  From    \eqref{coercivity-discreteBilinear},  \eqref{err-statd0}, the trace inequality, and \eqref{poincare-inequality}, we obtain
\begin{align*}
|||y_h^*-y^h|||^2 &\lesssim a_h(y_h^*-y^h,y_h^*-y^h)= (u_h^*-u^*,\kappa_2 (y_{h,1}^*-y^h_1)+\kappa_1 (y_{h,2}^*-y^h_2))_\Gamma \\
&\leq \|u^*-u_h^*\|_{0,\Gamma} \sum_{i=1}^2 \|y_{h,i}^*-y_i^h\|_{0,\Gamma} \\
&\lesssim \|u^*-u_h^*\|_{0,\Gamma}|||y_h^*-y^h|||
\end{align*}
which, together with the triangle inequality, yields \eqref{0error-meshnorm-state}.

Finally, let us show \eqref{0error-meshnorm-costate}. From   \eqref{coercivity-discreteBilinear},  \eqref{err-adjod0}, and  \eqref{poincare-inequality},  we get
\begin{align*}
|||p_h^*-p^h|||^2&\lesssim a_h(p_h^*-p^h,p_h^*-p^h)=(y_h^*-y^*, p_h^*-p^h)_\Omega\\
  &\lesssim ||y_h^*-y^*||_{0,\Omega}|| p_h^*-p^h||_{0,\Omega}\\
  &\lesssim ||y_h^*-y^*||_{0,\Omega}||| p_h^*-p^h|||,
\end{align*}
which, together with the triangle inequality, indicates  \eqref{0error-meshnorm-costate}.
\end{proof}

Based on above theorem, with the help of \eqref{errorMeshnorm-cutH2regularity-state}-\eqref{errorL2norm-cutH2regularity-costate}, \eqref{error-meshnorm-state-lowRegularity}-\eqref{error-l2norm-costate-lowRegularity} and \eqref{errorInterface-costate-lowRegularity},  we can immediately obtain the following  main results of optimal error estimates.
\begin{theorem}
Let $(y^*,p^*,u^*)\in \left(H_0^1(\Omega)\cap H^s (\Omega_1 \cup \Omega_2)\right)\times \left(H_0^1(\Omega)\cap H^s (\Omega_1 \cup \Omega_2)\right)\times U_{ad}$ and $(y_h^*,p_h^*,u_h^*)\in V^h\times V^h\times U_{ad} (s=2,3/2)$ be  the solutions to the continuous problem \eqref{stateEqu2-continue}-\eqref{variational-inequality-continue} and the discrete problem \eqref{stateEqu2-discrete}-\eqref{variational-inequality-discrete}, respectively. Then we have, for $s=2$,
\begin{align}
\|u^*-u_h^*\|_{0,\Gamma}+\|y^*-y_h^*\|_{0,\Omega} +\|p^*-p_h^*\|_{0,\Omega} &\lesssim h^\frac{3}{2} (\|y^*\|_{2,\Omega_1\cup \Omega_2}+\|p^*\|_{2,\Omega_1\cup \Omega_2}),\label{error-l2-interContr}\\
|||y^*-y_h^*||| + |||p^*-p_h^*|||&\lesssim h (\|y^*\|_{2,\Omega_1\cup \Omega_2}+\|p^*\|_{2,\Omega_1\cup \Omega_2}), \label{error-meshnorm-interContr}
\end{align}
and for $s=3/2$,
\begin{align}
\|u^*-u_h^*\|_{0,\Gamma}+\|y^*-y_h^*\|_{0,\Omega} +\|p^*-p_h^*\|_{0,\Omega} &\lesssim h (\|y^*\|_{\frac32,\Omega_1\cup \Omega_2}+\|p^*\|_{\frac32,\Omega_1\cup \Omega_2}),\label{error-l2-lowRegularity-interContr}\\
|||y^*-y_h^*||| + |||p^*-p_h^*|||&\lesssim h^\frac{1}{2} (\|y^*\|_{\frac32,\Omega_1\cup \Omega_2}+\|p^*\|_{\frac32,\Omega_1\cup \Omega_2}). \label{error-meshnorm-lowRegularity-interContr}
\end{align}
\end{theorem}

\section{Numerical results}

We shall provide several 2D numerical examples to verify the performance of the proposed finite element method. Because the variational inequality (\ref{variational-inequality-discrete}) is just equivalent to a projection, we shall simply use the fixed-point iteration algorithm to compute the discrete optimality problem \eqref{stateEqu2-discrete}-\eqref{variational-inequality-discrete}.
\paragraph{Algorithm}
\begin{enumerate}
\item Initialize $u^i_h=u^0$;
\item Compute $y^i_h\in V^h$ by $a_h(y_h^i,w_h)=(f,w_h)+(g+u_h^i,\kappa_2 w_{h,1}+\kappa_1 w_{h,2})_\Gamma,\forall w_h\in V^h$; \label{computeyhi}
\item Compute $p_h^i\in V^h$ by $a_h(w_h,p_h^i)=(y_h^i-y_d,w_h),\forall w_h\in V^h$;
\item Set $u_h^{i+1}=\max\{u_a,\min\{-\frac{\kappa_2 p_{h,1}^i+\kappa_1 p_{h,2}^i}{\alpha},u_b\}\}$;
\item if $|u_h^{i+1}-u_h^i|<\text{Tol}$ or $i+1>\text{MaxIte}$, then output $u_h^*=u_h^{i+1}$, else $i=i+1$, and go back to Step \ref{computeyhi}.
\end{enumerate}
Here $u^0$ is an initial value, Tol is the tolerance, and MaxIte is the maximal iteration number. This algorithm is convergent when the regularity parameter $\alpha$ is large enough (cf. \cite{Hinze09opt}).

In all numerical examples, we choose $\Omega \subseteq \mathbb{R}^2$ to be a square, and use $N\times N$  uniform  meshes with  $2N^2$ triangular elements. 

\begin{example}
Segment interface.

Take $\Omega:=[0,1]\times [0,1]$ (cf. Figure \ref{ex1interface-figure}) with a segment interface $$\Gamma:=\{(x_1,x_2):x_2=kx_1+b\}\cap \Omega,$$ where $k=-\sqrt{3}/3, b=(6+\sqrt{6}-2\sqrt{3})/6$, and set $$\Omega_1:=\{(x_1,x_2):x_2>kx_1+b\}\cap \Omega,\quad  \Omega_2:=\{(x_1,x_2):x_2<kx_1+b\}\cap \Omega.$$
Choose $\ U_{ad}=\{v\in L^2(\Gamma):\sin(\pi (x_1-1/2))\leq v\leq 1, \text{ a.e. on } \Gamma\}$,
\[a(x_1,x_2)=\left\{
\begin{array}{ll}
1, & \text{if }(x_1,x_2)\in \Omega_1, \\
100, & \text{if }(x_1,x_2)\in \Omega_2.
\end{array}
\right.\]
Let $y_d, f, g$ be such that  the optimal triple $(y^*,p^*,u^*)$  of optimal control problem \eqref{stateEqu2-continue}-\eqref{variational-inequality-continue}  is defined as follows
\begin{align*}
 & y^*(x_1,x_2)=\left\{
\begin{array}{lr}
(x_2-kx_1-b)\cos(x_1 x_2), & \text{if }(x_1,x_2)\in \Omega_1, \\
(x_2-kx_1-b)\cos(x_1 x_2)/100, & \text{if }(x_1,x_2)\in \Omega_2,
\end{array}
\right. \\
 & p^*(x_1,x_2)=\left\{
\begin{array}{lr}
100(x_2-kx_1-b)x_1(x_1-1)x_2(x_2-1)\sin(x_1 x_2), & \text{if }(x_1,x_2)\in \Omega_1, \\
(x_2-kx_1-b)x_1(x_1-1)x_2(x_2-1)\sin(x_1 x_2), & \text{if }(x_1,x_2)\in \Omega_2,
\end{array}
\right. \\
 &u^*(x_1,x_2)=\max\{\sin(\pi(x_1-1/2)),0\} \text{ for } (x_1,x_2)\in \Gamma.
\end{align*}
\end{example}

\begin{figure}[htbp]
\centering
\includegraphics[width=4cm]{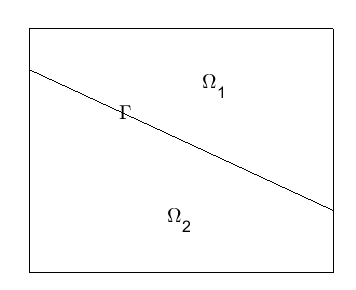}
\caption{Segment interface for Example 5.1}
\label{ex1interface-figure}
\end{figure}

We compute the discrete schemes  \eqref{stateEqu2-discrete}-\eqref{variational-inequality-discrete} with  the regularity parameter $\alpha =1, 0.0001 $ and the stabilization parameter $\widetilde{C} =50, 1000$. We note that, from (3.10), $\widetilde{C}$ is required to be sufficiently large to keep  the coerciveness of $a_h (\cdot,\cdot)$. Tables \ref{errorL2-reg1stab50-ex1}-\ref{error-reg1stab1000-ex1} show the history of convergence for the optimal discrete triple $(y_h^*,p_h^*,u_h^*)$, where for simplicity we set $|\cdot|_1:=|\cdot|_{1,\Omega_1\cup\Omega_2}, \ \|\cdot\|_0:=\|\cdot\|_{0,\Omega}$. For comparison, we also list in Tables \ref{errorL2-reg1stab50-ex1}-\ref{errorH1-reg1stab50-ex1} the results obtained by using the conforming linear finite element method ($P_1$-FEM).

From the numerical results, we can see that for all cases the CutFEM  yields first order rates of convergence for $|y^*-y_h^*|_1$ and $|p^*-p_h^*|_1$, which are consistent with the theoretical results \eqref{error-l2-interContr}-\eqref{error-meshnorm-interContr}, and yields second  rates of convergence for $\|y^*-y_h^*\|_0$, $\|u^*-u_h^*\|_{0,\Gamma}$ and $\|p^*-p_h^*\|_0$, which are better than the theoretical order $3/2$. We can also see that,
without using interface-fitted meshes and adding  into the approximation additional basis functions characterizing the singularity   around the interface, the $P_1$-FEM is not able to attain  optimal convergence.

\begin{table}[htbp]
	\centering
	\begin{tabular}{|c|c|c|c|c|c|c|c|}
		\hline
		&N & $\|y^*-y_h^*\|_{0 }$ &order& $\|u^*-u_h^*\|_{0,\Gamma}$&order& $\|p^*-p_h^*\|_{0}$ &order  \\
		\hline
	&16&1.79e-2& &1.47e-3& &2.94e-2& \\
	&32& 9.21e-3  & 1.0& 6.98e-4  &1.1 &1.31e-2&1.2\\
	$P_1-FEM$ &64& 4.68e-3  & 1.0& 3.21e-4 &1.1&6.22e-3&1.1\\
	&128& 2.36e-3 & 1.0&1.57e-4 &1.0&3.03e-3 &  1.0\\
	&256& 1.18e-3  & 1.0&7.68e-5 &1.0&1.50e-3&1.0 \\
		\hline
	& 16 & 5.30e-4  & &5.51e-4   & &1.52e-2 &\\
	&32 &1.32e-4  &2.0 &1.20e-4     &2.2& 2.95e-3 &2.4\\
	CutFEM&64 & 3.31e-5  &2.0 &2.47e-5    &2.3& 6.81e-4&2.1 \\
	&128&8.28e-6  &2.0 &5.62e-6    & 2.1& 1.49e-4&2.2 \\
	&256&2.06e-6  &2.0 &1.37e-6    & 2.0& 3.41e-5& 2.1\\
		\hline
	\end{tabular}
	\caption{History of convergence  in $L^2$-norm (Example 5.1): $\alpha=1, \widetilde{C}=50$}\label{errorL2-reg1stab50-ex1}
\end{table}

\begin{table}[htbp]
	\centering
	\begin{tabular}{|c|c|c|c|c|c|}
		\hline
		&N & $|y^*-y_h^*|_{1 }$ &order&$|p^*-p_h^*|_{1 }$ &order  \\
		\hline
	&16&	2.74e-1& &4.99e-1& \\
	&32&	1.90e-1 &0.5& 2.91e-1  & 0.8\\
	$P_1-FEM$ &64&	1.32e-1 &0.5 &  1.79e-1 & 0.7\\
	&128&9.28e-2 &0.5& 1.16e-1  &0.6 \\
	&256&6.51e-2 &0.5&  7.83e-2 & 0.6\\
		\hline
	& 16 &3.12e-2 &&   3.81e-1    & \\
	&32 &1.56e-2  & 1.0& 1.84e-1   &1.1\\
	CutFEM&64 &7.81e-3  & 1.0 & 9.00e-2  &1.0 \\
	&128&3.90e-3  & 1.0&4.44e-2 & 1.0 \\
	&256&1.95e-3  &1.0 &2.21e-2 & 1.0 \\
		\hline
	\end{tabular}
	\caption{History of convergence  in $H^1$-seminorm  (Example 5.1): $\alpha=1, \widetilde{C}=50$}\label{errorH1-reg1stab50-ex1}
\end{table}


\begin{table}
	\centering
	\begin{tabular}{|c|c|c|c|c|c|c|c|c|c|c|}
		\hline
		N&$|y-y_h|_1$&order &$\|y-y_h\|_0$&order&$\|u-u_h\|_0$&order&$|p-p_h|_1$&order&$\|p-p_h\|_0$&order \\
		\hline
16 & 3.12e-2& &  5.32e-4  & &  1.78e-7 &&  5.54e-5 & & 4.14e-6& \\
32 &1.56e-2 & 1.0& 1.33e-4 &2.0   & 3.61e-8 & 2.3 & 2.14e-5 & 1.4& 1.16e-6&1.8 \\
64 &7.81e-3&1.0 & 3.36e-5 &2.0 & 1.02e-8& 1.8& 9.44e-6 &1.2  & 3.13e-7&1.8 \\
128&3.90e-3 & 1.0 & 8.49e-6 &2.0 & 2.70e-9 &1.9& 4.50e-6 & 1.1 & 8.39e-8&1.9 \\
256 & 1.95e-3 &1.0 &  2.12e-6 & 2.0 & 6.98e-10  &2.0& 2.22e-6  &1.0& 2.16e-8&2.0 \\
		\hline
		\end{tabular}	\caption{History of convergence for CutFEM (Example 5.1): $\alpha =0.0001,\widetilde{C}=50$}
	\label{error-reg00001stab50-ex1}
\end{table}


\begin{table}
	\centering
	\begin{tabular}{|c|c|c|c|c|c|c|c|c|c|c|}
		\hline
		N &$|y-y_h|_1$&order&$\|y-y_h\|_0$&order&$\|u-u_h\|_0$&order&$|p-p_h|_1$&order&$\|p-p_h\|_0$&order\\
		\hline
	 16&3.15e-2 & &5.66e-4  && 7.79e-4   &&  4.24e-1 &  & 2.14e-2 &\\
	32&1.58e-2 &1.0& 1.57e-4 &1.9 & 2.28e-4  &1.8 & 2.02e-1  & 1.1  & 5.01e-3&2.1 \\
	64 &7.86e-3 &1.0&  3.81e-5  & 2.0& 6.12e-5  & 1.9 &  1.00e-1& 1.0  & 1.47e-3& 1.8\\
	128&3.91e-3  &1.0& 9.05e-6  &2.1& 1.69e-5 & 1.9 & 4.80e-2  &1.1 &  3.21e-4 &2.2\\
	256&1.96e-3 &1.0& 2.14e-6  &2.1& 3.74e-6  & 2.2& 2.29e-2  &1.1 & 6.53e-5 &2.3\\
			\hline
	\end{tabular}
	\caption{History of convergence for CutFEM (Example 5.1): $\alpha =1,\widetilde{C}=1000$}
	\label{error-reg1stab1000-ex1}
\end{table}

\begin{example}
Polygonal line interface.

Take $\Omega:=[0,2]\times [0,2]$ (cf. Figure \ref{ex2interface-figure}) with a  polygonal line interface
$$\Gamma:=\{(x_1,x_2):\varphi(x_1,x_2)=0,b\leq x_1 \leq 2-b,b\leq x_2 \leq 2-b\},$$ where $\varphi(x_1,x_2)=(x_2-(-x_1+1+b))(x_2-(x_1-1+b))(x_2-(-x_1-b+3))(x_2-(x_1+1-b)),b=\sqrt{3}/4$. And set
$$\Omega_1:=\{(x_1,x_2):\varphi(x_1,x_2)>0,b\leq x_1 \leq 2-b,b\leq x_2 \leq 2-b\}, \quad \Omega_2:=\Omega \setminus \{\Omega_1 \cup \Gamma\}.$$
Take $\alpha=1, U_{ad}:=\{v\in L^2(\Gamma):\sin (2\pi x_1)\leq v\leq 1, \text{ a.e. on } \Gamma\}$,
\[a(x_1,x_2)=\left\{
\begin{array}{ll}
1, & \text{if }(x_1,x_2)\in \Omega_1, \\
10, & \text{if }(x_1,x_2)\in \Omega_2.
\end{array}
\right.\]
Let $y_d, f, g$ be such that  the optimal triple $(y^*,p^*,u^*)$ of optimal control problem \eqref{stateEqu2-continue}-\eqref{variational-inequality-continue} is defined as follows
\begin{align*}
 & y^*(x_1,x_2)=\left\{
\begin{array}{lr}
10\varphi(x_1,x_2)e^{(x_1-1)(x_2-1))}, & \text{if }(x_1,x_2)\in \Omega_1, \\
\varphi(x_1,x_2)e^{(x_1-1)(x_2-1))}, & \text{if }(x_1,x_2)\in \Omega_2,
\end{array}
\right. \\
 & p^*(x_1,x_2)=\left\{
\begin{array}{lr}
10\varphi(x_1,x_2)x_1(x_1-2)x_2(x_2-2), & \text{if }(x_1,x_2)\in \Omega_1, \\
\varphi(x_1,x_2)x_1(x_1-2)x_2(x_2-2), & \text{if }(x_1,x_2)\in \Omega_2,
\end{array}
\right. \\
 & u^*(x_1,x_2)=\max\{\sin(2\pi x_1),0\}, \text{  for } (x_1,x_2)\in \Gamma.
\end{align*}
Notice that $y^*,p^* \notin  H^{2}(\Omega_1 \cup \Omega_2)$, but $y^*,p^* \in  H^{3/2}(\Omega_1 \cup \Omega_2)$.
\end{example}

\begin{figure}[htbp]
\centering
\includegraphics[width=4cm]{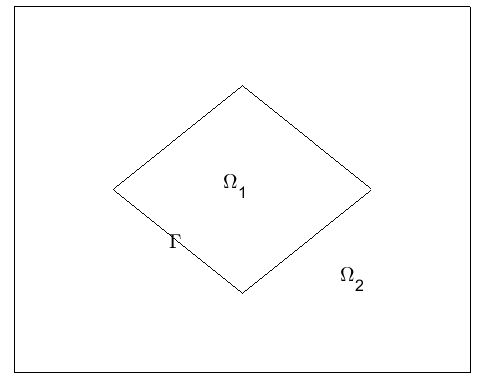}
\caption{Polygonal line interface for Example 5.2}
\label{ex2interface-figure}
\end{figure}

We compute the discrete schemes  \eqref{stateEqu2-discrete}-\eqref{variational-inequality-discrete} with the stabilization parameter $\widetilde{C}$ of \eqref{stablization-constant} as $\widetilde{C}=50$. Table \ref{cutError-ex2} shows the history of convergence for the optimal discrete triple $(y_h^*,p_h^*,u_h^*)$.

From the numerical results, we can see that the  CutFEM  shows higher order rates of convergence than the theoretical results \eqref{error-l2-lowRegularity-interContr}-\eqref{error-meshnorm-lowRegularity-interContr} (with $s=3/2$) for all the error terms. We note that our numerical results are  also better than those in   \cite[Table 1 and Table 2]{Wachsmuth16opt}, which are roughly $ (\delta+0.5)$-order  for $\|u^*-u_h^*\|_{0,\Gamma}$, $ (2\delta)$-order for $\|y^*-y_h^*\|_{0}$, and $\delta$-order for $|y^*-y_h^*|_{1}$ with $\delta=0.7$.

\begin{table}[htbp]
	\centering
\footnotesize
	\begin{tabular}{|c|c|c|c|c|c|c|c|c|c|c|}
		\hline
		N& $|y^*-y_h^*|_1$ &order & $\|y^*-y_h^*\|_0$ &order& $\|u^*-u_h^*\|_{0,\Gamma}$&order&$|p^*-p_h^*|_1$ &order& $\|p^*-p_h^*\|_0$ &order\\
		\hline
		16&	1.04& &3.49e-2& &7.25e-3& & 6.23e-1& &3.43e-2&\\
		32&	4.95e-1 &1.1 &6.46e-3  &2.4 &8.58e-4  &3.0& 2.67e-1  &1.2&5.79e-3&2.5\\
		64&	2.50e-1 &1.0&1.81e-3  &1.8& 4.13e-4 &1.1&  1.35e-1 &1.0&1.58e-3&1.9\\
		128&1.26e-2 &1.0&5.19e-4 &1.8&1.70e-4 &1.3& 6.75e-2  &1.0&4.21e-4 & 1.9\\
		256&6.25e-2 & 1.0&1.31e-4  &2.0&5.15e-5 &1.7&3.33e-2 &1.0&8.38e-5&2.3\\
		\hline
	\end{tabular}
	\caption{History of convergence for CutFEM (Example 5.2)}
	\label{cutError-ex2}
\end{table}

\begin{example}
Five-star interface.

Take $\Omega:=[-1,1]\times [-1,1]$ (cf. Figure \ref{ex3interface-figure}) with a 5-star  interface $$\Gamma:=\{(x_1,x_2):\varphi(r,\theta)=0,0\leq \theta \leq 2\pi\},$$
where $\varphi(x_1,x_2)=r-\frac{\sqrt{3}}{4}-0.1\sin(5\theta+\frac{\pi}{2})$, with $x_1=r \cos\theta, x_2=r \sin\theta$. And set $$\Omega_1:=\{(x_1,x_2):\varphi(x_1,x_2)<0\}\cap \Omega, \quad \Omega_2:=\Omega \setminus \{\Omega_1 \cup \Gamma\}.$$
Take $\alpha=1, U_{ad}:=\{v\in L^2(\Gamma):0\leq v\leq 1\}$,
\[a(x_1,x_2)=\left\{
\begin{array}{ll}
1, & \text{if }(x_1,x_2)\in \Omega_1, \\
10, & \text{if }(x_1,x_2)\in \Omega_2,
\end{array}
\right.\]
\[g=0,~y_d=\left\{
\begin{array}{ll}
10, & \text{if }(x_1,x_2)\in \Omega_1, \\
1, & \text{if }(x_1,x_2)\in \Omega_2,
\end{array}
\right.~f=1.
\]
Since the interface $\Gamma$ is of complicated shape,  it is difficult to give the explicit expressions of  the optimal triple $(y^*,p^*,u^*)$.
\end{example}

\begin{figure}[htbp]
\centering
\includegraphics[width=4cm]{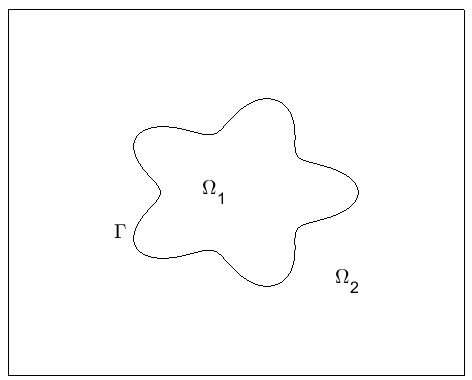}
\caption{Five-star interface for Example 5.3}
 \label{ex3interface-figure}
\end{figure}

We compute the discrete schemes \eqref{stateEqu2-discrete}-\eqref{variational-inequality-discrete} with the stabilization parameter $\widetilde{C}=50$ and $1000$. Let $y_{h,50}^*$ and  $y_{h,1000}^*$ denote the CutFEM approximations    of state $y$ with $\widetilde{C}=50$ and $\widetilde{C}=1000$, respectively. Also let $p_{h,50}^*$ and and $p_{h,1000}^*$ denote the CutFEM approximations of co-state $p$ with $\widetilde{C}=50$ and  $\widetilde{C}=1000$, respectively.

In Figures \ref{CutfemEx3-50-figures}-\ref{CutfemEx3-1000-figures}, we give the optimal discrete states   $y_{h,50}^*$, $y_{h,1000}^*$, and the discrete co-states   $p_{h,50}^*$, $p_{h,1000}^*$ on $64\times 64$ mesh. 
  Figure \ref{Difference-CutfemEx3-figures} demonstrates   the difference $y_{h,1000}^*-y_{h,50}^*$ and $p_{h,1000}^*-p_{h,50}^*$  on $64\times 64$ mesh. These figures show that  the numerical interfaces are distinct  for both the state and co-state and accord with the interface of the equations.  Once again we find that, a large $\widetilde{C}$ may affect the numerical results slightly. 

\begin{figure}[htbp]
 \centering
 \begin{minipage}{10cm}
 \includegraphics[width=4cm]{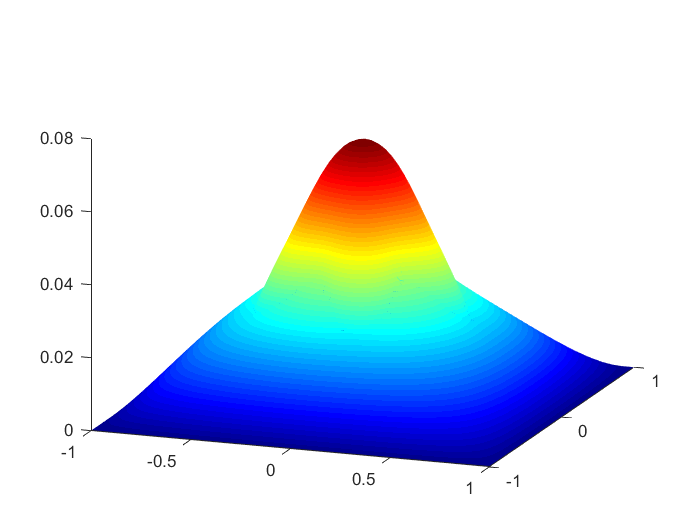}
 \includegraphics[width=4cm]{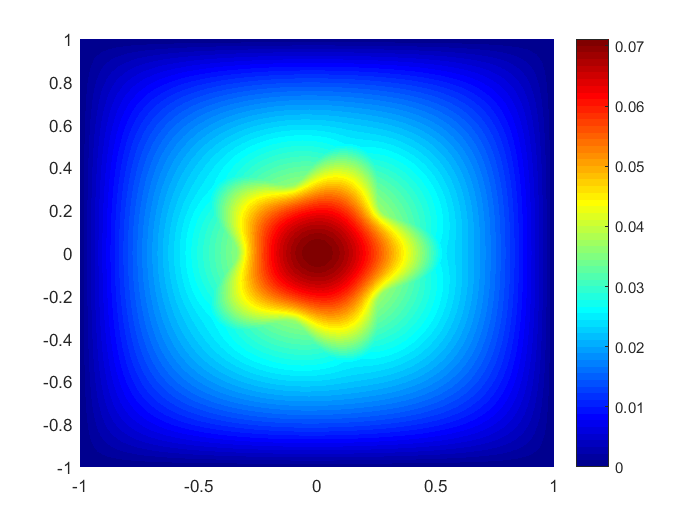}\\
  \end{minipage}
\qquad\qquad
   \begin{minipage}{10cm}
 \includegraphics[width=4cm]{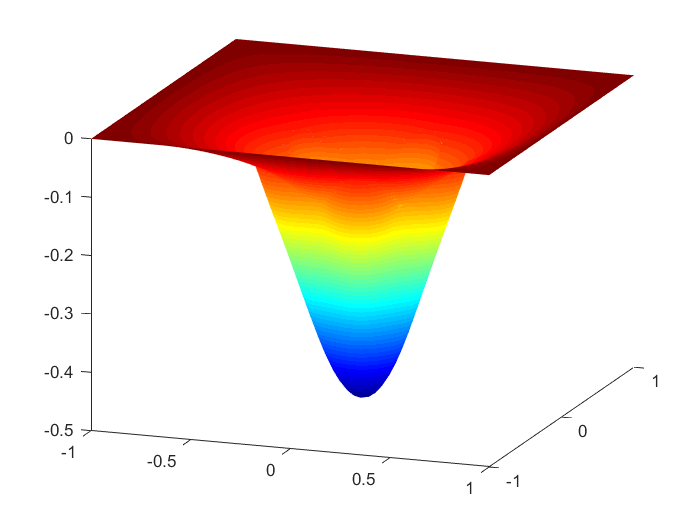}
 \includegraphics[width=4cm]{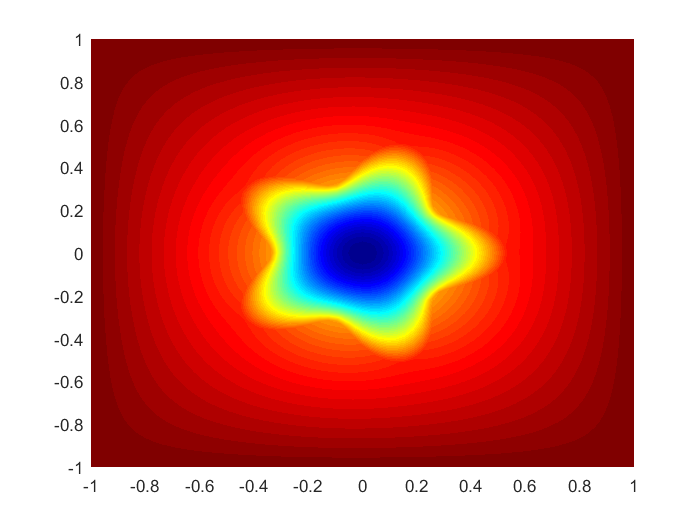}
  \end{minipage}
\caption{The CutFEM approximations (Example 5.3): $\widetilde{C}=50$. The upper    two figures  and the lower two figures show the graphs of  $y_{h,50}^*$ and $p_{h,50}^*$, respectively.} 
  \label{CutfemEx3-50-figures}
\end{figure}

\begin{figure}[htbp]
 \centering
 \begin{minipage}{10cm}
 \includegraphics[width=4cm]{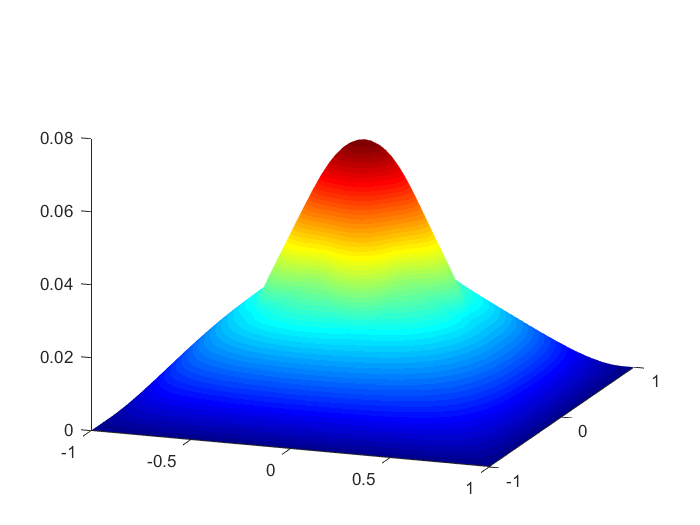}
 \includegraphics[width=4cm]{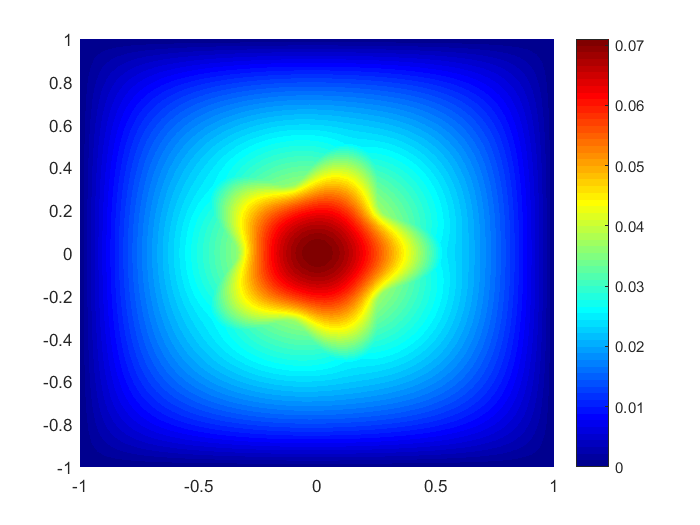}
  \end{minipage}
\qquad\qquad
   \begin{minipage}{10cm}
 \includegraphics[width=4cm]{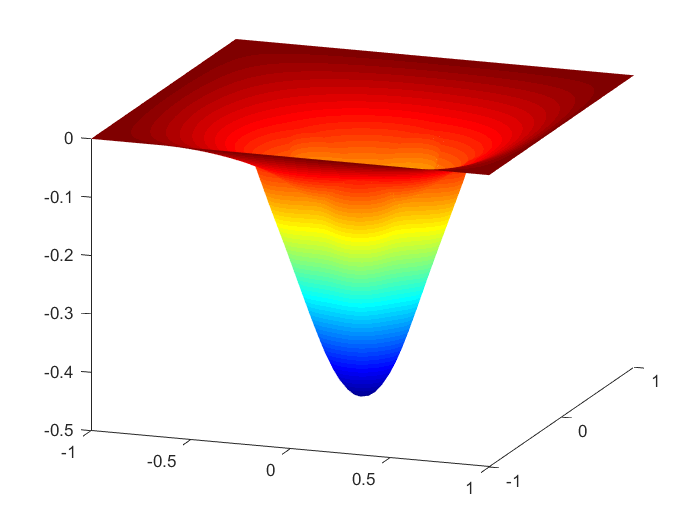}
 \includegraphics[width=4cm]{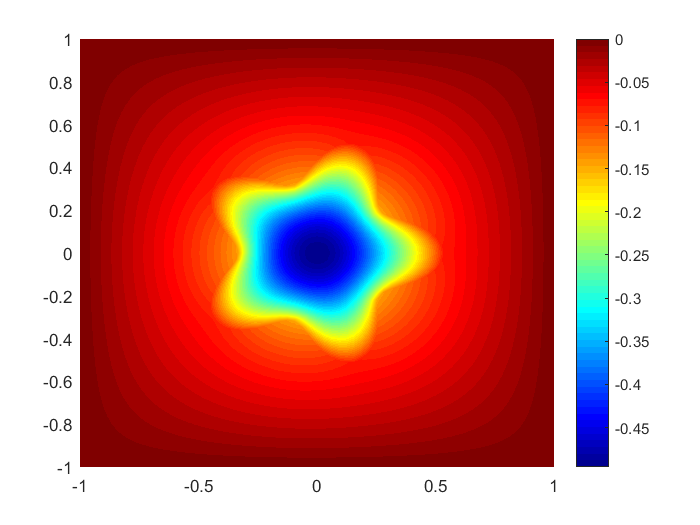}
  \end{minipage}
\caption{The CutFEM approximations (Example 5.3): $\widetilde{C}=1000$. The upper    two figures  and the lower two figures show the graphs of  $y_{h,1000}^*$ and $p_{h,1000}^*$, respectively.} 
  \label{CutfemEx3-1000-figures}
\end{figure}

\begin{figure}[htbp]
 \centering
 \begin{minipage}{12cm}
 \includegraphics[width=6cm]{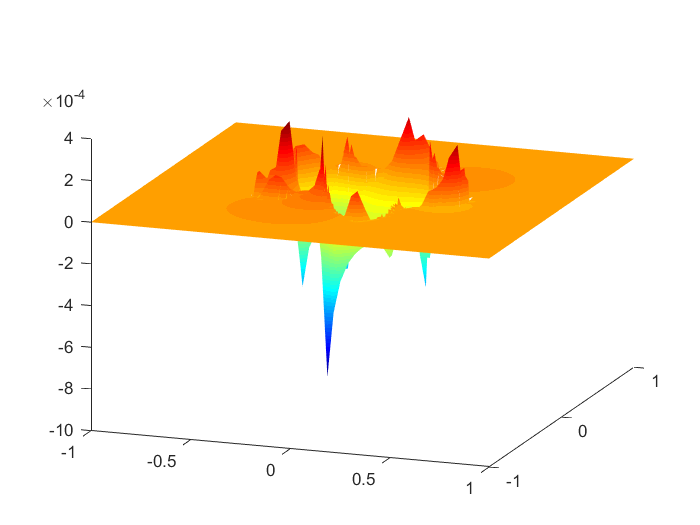}
 \includegraphics[width=6cm]{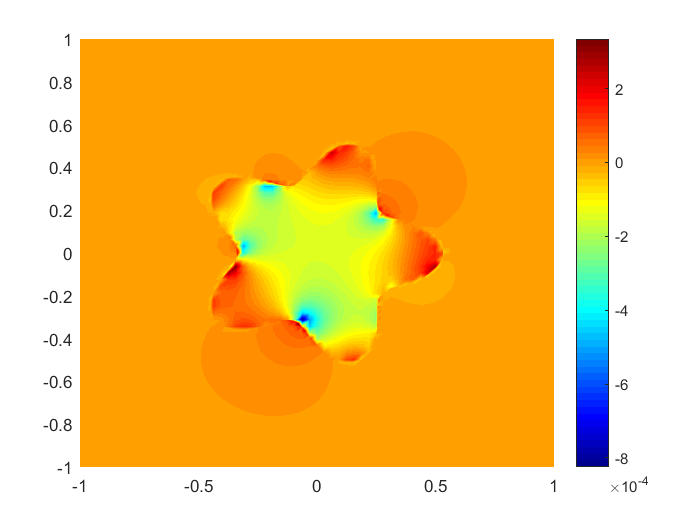}
  \end{minipage}
\qquad\qquad
   \begin{minipage}{12cm}
 \includegraphics[width=6cm]{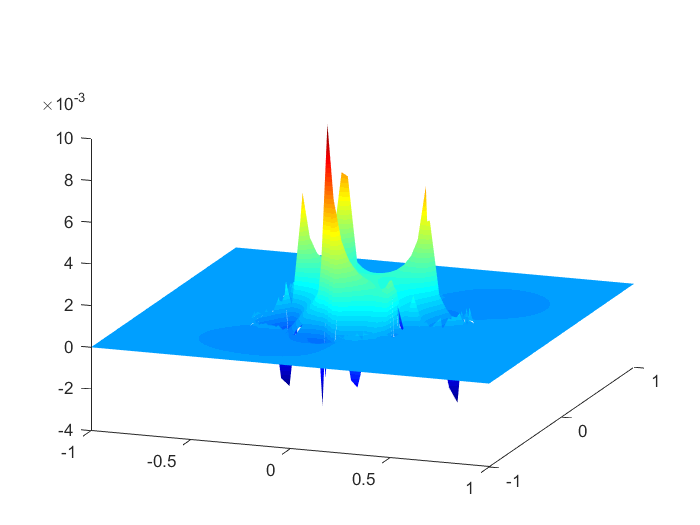}
 \includegraphics[width=6cm]{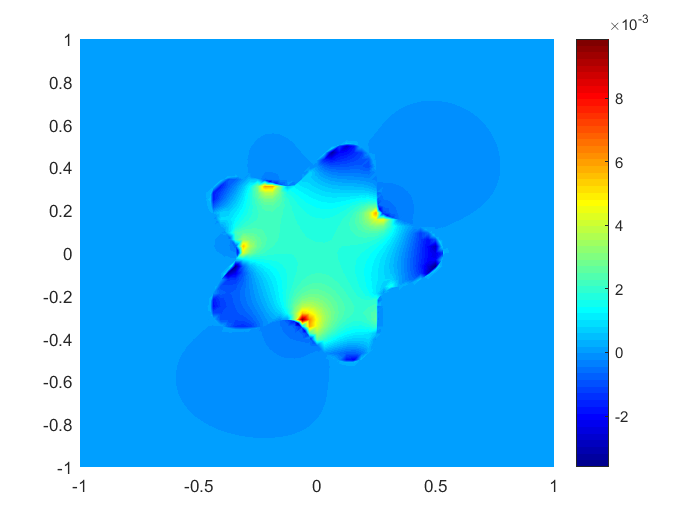}
  \end{minipage}
\caption{The difference between the CutFEM approximations (Example 5.3):  The upper    two figures  and the lower two figures show the graphs of  $y_{h,1000}^*-y_{h,50}^*$ and $p_{h,1000}^*-p_{h,50}^*$, respectively.} 
  \label{Difference-CutfemEx3-figures}
\end{figure}

\bibliography{InterfaceControlBib}

\begin{thebibliography}{10}

\bibitem{Babuska70fin}
I.~Babu\u{s}ka.
\newblock The finite element method for elliptic equations with discontinuous
  coefficients.
\newblock {\em Computing}, 5:207--213, 1970.

\bibitem{Barretts87Fit}
J.W. Barrett and C.M. Elliott.
\newblock Fitted and unfitted finite-element methods for elliptic equations
  with smooth interfaces.
\newblock {\em IMA J. Numer. Anal.}, 7(3):283--300, 1987.

\bibitem{Beckers09nit}
R.~Becker, E.~Burman, and P.~Hansbo.
\newblock A {N}itsche extended finite element method for incompressible
  elasticity with discontinuous modulus of elasticity.
\newblock {\em Comput. Methods Appl. Mech. Engrg.}, 198(41-44):3352--3360,
  2009.

\bibitem{Beckers00ada}
R.~Becker, H.~Kapp, and R.~Rannacher.
\newblock Adaptive finite element methods for optimal control of partial
  differential equations: Basic concept.
\newblock {\em SIAM J. Control Optim.}, 39(1):113--132, 2000.

\bibitem{Belyschkos09rev}
T.~Belytschko, R.~Gracie, and G.~Ventura.
\newblock A review of extended/generalized finite element methods for material
  modeling.
\newblock {\em Modelling Simul. Mater. Sci. Eng.}, 17(4):1--24, 2009.

\bibitem{Benedixs09pos}
O.~Benedix and B.~Vexler.
\newblock A posteriori error estimation and adaptivity for elliptic optimal
  control problems with state constraints.
\newblock {\em Comput. Optim. Appl.}, 44(1):3--25, 2009.

\bibitem{Brambles96fin}
J.H. Bramble and J.T. King.
\newblock A finite element method for interface problems in domains with smooth
  boundaries and interfaces.
\newblock {\em Adv. Comput. Math.}, 6(1):109--138, 1996.

\bibitem{Bren08}
S.C. Brenner and L.R. Scott.
\newblock {\em The mathematical theory of finite element methods}.
\newblock Springer-Verlag, Berlin, 3rd edition, 2008.

\bibitem{Brezis11fun}
H.~Brezis.
\newblock {\em Functional analysis, {S}obolev space and partial differential
  equations}.
\newblock Springer, 2011.

\bibitem{Burman15cut}
E.~Burman, S.~Claus, P.~Hansbo, M.G. Larson, and A.~Massing.
\newblock Cutfem: discretizing geometry and partial differential equations.
\newblock {\em Int. J. Numer. Meth. Engng}, 104(7):472--501, 2015.

\bibitem{Cai17dis}
Z.~Cai, C.~He, and S.~Zhang.
\newblock Discontinuous finite element methods for interface problems: Robust a
  priori and a posteriori error estimates.
\newblock {\em SIAM J. Numer. Anal.}, 55(1):400--418, 2017.

\bibitem{Camps06qua}
B.~Camp, T.~Lin, Y.~Lin, and W.~Sun.
\newblock Quadratic immersed finite element spaces and their approximation
  capabilities.
\newblock {\em Adv. Comput. Math.}, 24(1-4):81--112, 2006.

\bibitem{Cenanovics16cut}
M.~Cenanovia, P.~Hansbo, and M.G. Larson.
\newblock Cut finite element modeling of linear membranes.
\newblock {\em Comput. Methods Appl. Mech. Engrg.}, 310:98--111, 2016.

\bibitem{Chens10err}
Y.~Chen, Y.~Huang, W.~Liu, and N.~Yan.
\newblock Error estimates and superconvergence of mixed finite elment methods
  for convex optimal control problems.
\newblock {\em J. Sci. Comput.}, 42(3):382--403, 2010.

\bibitem{Chen98}
Z.~Chen and J.~Zou.
\newblock Finite element methods and their convergence for elliptic and
  parabolic interface problems.
\newblock {\em Numer. Math.}, 79(2):175--202, 1998.

\bibitem{Deka18wea}
B.~Deka.
\newblock A weak galerkin finite element method for elliptic interface problems
  with polynomial reduction.
\newblock {\em Numer. Math. Theor. Meth. Appl.}, 11:655--672, 2018.

\bibitem{Fedkiw06imm}
R.~Fedkiw.
\newblock The immersed interface method. numerical solutions of pdes involving
  interfaces and irregular domains.
\newblock {\em Math. Comput.}, 76(259):1691, 2006.

\bibitem{Gongs16ada}
W.~Gong and N.~Yan.
\newblock Adaptive finite element method for elliptic optimal control problems:
  convergence and optimality.
\newblock {\em Numer. Math.}, 135(4):1121--1170, 2017.

\bibitem{Hans16a3d}
D.~Han, P.~Wang, X.~He, T.~Lin, and J.~Wang.
\newblock A 3{D} immersed finite element method with non-homogeneous interface
  flux jump for applications in particle-in-cell simulations of plasma-lunar
  surface interactions.
\newblock {\em J. Comput. Phys.}, 321:965--980, 2016.

\bibitem{Hans02}
A.~Hansbo and P.~Hansbo.
\newblock An unfitted finite element method, based on nitsche's method, for
  elliptic interface problems.
\newblock {\em Comput. Methods Appl. Mech. Engrg.}, 191(47-48):5537--5552,
  2002.

\bibitem{Hansbo14cut}
P.~Hansbo, M.G. Larson, and S.~Zahedi.
\newblock A cut finite element method for a stokes interface problem.
\newblock {\em Appl. Numer. Math.}, 85(C):90--114, 2014.

\bibitem{Hes11imm}
X.~He, T.~Lin, and Y.~Lin.
\newblock Immersed finite element methods for elliptic interface problems with
  non-homogeneous jump conditions.
\newblock {\em Int. J. Numer. Anal. Model.}, 8(2):284--301, 2011.

\bibitem{Hes12con}
X.~He, T.~Lin, and Y.~Lin.
\newblock The convergence of the bilinear and linear immersed finite element
  solutions to interface problems.
\newblock {\em Numer. Methods Partial Differential Equations}, 28(1):312--330,
  2012.

\bibitem{Hintermuller10goa}
M.~Hinterm\"{u}ller and R.H.W. Hoppe.
\newblock Goal-oriented adaptivity in pointwise state constrained optimal
  control of partial differential equations.
\newblock {\em SIAM J. Control Optim.}, 48(8):5468--5487, 2010.

\bibitem{Hinz05var}
M.~Hinze.
\newblock A variational discretization concept in control constrained
  optimization: The linear-quadratic case.
\newblock {\em Comput. Optim. Applic.}, 30:45--61, 2005.

\bibitem{Hinz09}
M.~Hinze and U.~Matthes.
\newblock A note on variational discretization of elliptic neumann boundary
  control.
\newblock {\em Control Cybernet.}, 38(3):577--591, 2009.

\bibitem{Hinze09opt}
M.~Hinze, R.~Pinnau, M.~Ulbrich, and S.~Ulbrich.
\newblock {\em Optimization with PDE Constraints}.
\newblock Springer, 2009.

\bibitem{Huangs02mor}
J.~Huang and J.~Zou.
\newblock A mortar element method for elliptic problems with discontinuous
  coefficients.
\newblock {\em IMA J. Numer. Anal.}, 22(4):549--576, 2002.

\bibitem{Jeri81}
D.S. Jerison and C.E. Kenig.
\newblock The {N}eumann problem on {L}ipschitz domains.
\newblock {\em Bull. Am. Math. Soc.}, 4(2):203--207, 1981.

\bibitem{Jis18par}
H.~Ji, Q.~Zhang, Q.~Wang, and Y.~Xie.
\newblock A partially penalised immersed finite element method for elliptic
  interface problems with non-homogeneous jump conditions.
\newblock {\em East. Asia. J. Appl. Math.}, 8:1--23, 2018.

\bibitem{Kohls14Convergence}
K.~Kohls, K.G. Siebert, and A.~R\"{o}sch.
\newblock Convergence of adaptive finite elements for optimal control problems
  with control constraints.
\newblock In G.~Leugering et~al., editor, {\em Trends in PDE Constrained
  Optimization}, volume 165 of {\em International Series of Numerical
  Mathematics}, pages 403--419. Birkh\"{a}user, Cham, 2014.

\bibitem{Lehrenfelds17opt}
C.~Lehrenfeld and A.~Reusken.
\newblock Optimal preconditioners for {Nitsche-XFEM} discretizations of
  interface problems.
\newblock {\em Numer. Math.}, 135(2):313--332, 2017.

\bibitem{Lis10opt}
J.~Li, M.J. Markus, B.~Wohlmuth, and J.~Zou.
\newblock Optimal a priori estimates for higher order finite elements for
  elliptic interface problems.
\newblock {\em Appl. Numer. Math.}, 60(1-2):19--37, 2010.

\bibitem{Lis02ada}
R.~Li, W.~Liu, H.~Ma, and T.~Tang.
\newblock Adaptive finite element approximation for distributed elliptic
  optimal control problems.
\newblock {\em SIAM J. Control Optim.}, 41(5):1321--1349, 2002.

\bibitem{Li06imm}
Z.~Li and K.~Ito.
\newblock {\em The immersed interface method: numerical solutions of PDEs
  involving interfaces and irregular domains}.
\newblock Frontiers Appl. Math. 33. SIAM, Philadelphia, 2006.

\bibitem{Li03new}
Z.~Li, T.~Lin, and X.~Wu.
\newblock New cartesian grid methods for interface problems using the finite
  element formulation.
\newblock {\em Numer. Math.}, 96(1):61--98, 2003.

\bibitem{Lins07err}
T.~Lin, Y.~Lin, and W.~Sun.
\newblock Error estimation of a class of quadratic immersed finite element
  methods for elliptic interface problems.
\newblock {\em Discrete Contin. Dyn. Syst. Ser. B}, 7(4):807--823, 2007.

\bibitem{Lins15par}
T.~Lin, Y.~Lin, and X.~Zhang.
\newblock Partially penalized immersed finite element methods for elliptic
  interface problems.
\newblock {\em SIAM J. Numer. Anal.}, 53(2):1121--1144, 2015.

\bibitem{Lion71}
J.L. Lions.
\newblock {\em Optimal control of systems governed by partial differential
  equations}.
\newblock Springer-Verlag, Berlin, 1971.

\bibitem{Lius17sec}
C.~Liu and C.~Hu.
\newblock A second order ghost fluid method for an interface problem of the
  poisson equation.
\newblock {\em Commun. Comput. Phys.}, 22(4):965--996, 2017.

\bibitem{Lius09new}
W.~Liu, W.~Gong, and N.~Yan.
\newblock A new finite element approximation of a state-constained optimal
  control problem.
\newblock {\em J. Comput. Math.}, 27(1):97--114, 2009.

\bibitem{Melenk95gen}
J.M. Melenk.
\newblock {\em On generalized finite element methods}.
\newblock PhD thesis, University of Maryland at College Park, 1995.

\bibitem{Melenks96par}
J.M. Melenk and I.~Babu\u{s}ka.
\newblock The partition of unity finite element method: basic theory and
  applications.
\newblock {\em Comput. Methods Appl. Mech. Engrg.}, 139(1-4):289--314, 1996.

\bibitem{Moes99fin}
N.~Mo\"{e}s, J.~Dolbow, and T.~Belytschko.
\newblock A finite element method for crack growth without remeshing.
\newblock {\em Int. J. Numer. Meth. Engng.}, 46(1):131--150, 1999.

\bibitem{Nitsche71ube}
J.~Nitsche.
\newblock \"{U}ber ein variationsprinzip zur l\"{o}sung von dirichlet-problemen
  bei verwendungvon von teilr\"{a}umen, die keinen randbedingungen unterworfen
  sind.
\newblock {\em Abh. Math. Univ. Hamburg}, 36(1):9--15, 1971.

\bibitem{Plums03opt}
M.~Plum and C.~Wieners.
\newblock Optimal a priori estimates for interface problems.
\newblock {\em Numer. Math.}, 95(4):735--759, 2003.

\bibitem{Roschs17rel}
A.~R\"{o}sch, K.G. Siebert, and S.~Steinig.
\newblock Reliable a posteriori error estimation for state-constrained optimal
  control.
\newblock {\em Comput. Optim. Appl.}, 68(1):121--162, 2017.

\bibitem{Schneiders16pos}
R.~Schneider and G.~Wachsmuth.
\newblock A posteriori error estimation for control-constrained,
  linear-quadratic optimal control problems.
\newblock {\em SIAM J. Numer. Anal.}, 54(2):1169--1192, 2016.

\bibitem{Schott17sta}
B.~Schott.
\newblock {\em Stabilized cut finite element methods for complex interface
  coupled flow problems}.
\newblock PhD thesis, Technische Universit\"{a}t M\"{u}nchen, 2017.

\bibitem{Scot90}
L.R. Scott and S.~Zhang.
\newblock Finite element interpolation of nonsmooth functions satisfying
  boundary conditions.
\newblock {\em Math. Comp.}, 54(190):483--493, 1990.

\bibitem{Strouboulis00des}
T.~Strouboulis, I.~Babu\u{s}ka, and K.~Copps.
\newblock The design and analysis of the generalized finite element method.
\newblock {\em Comput. Methods Appl. Mech. Engrg.}, 181(1):43--69, 2000.

\bibitem{Troltzsch10opt}
F.~Tr\"{o}ltzsch.
\newblock {\em Optimal control of partial differential equations: Theory,method
  and applications}, volume 112 of {\em Graduate studies in mathematics}.
\newblock American mathematical society, Providence, {R}hode {I}sland, 2010.

\bibitem{Wachsmuth16opt}
D.~Wachsmuth and J.E. Wurst.
\newblock Optimal control of interface problems with hp-finite elements.
\newblock {\em Nume. Funct. Anal. Optim.}, 37(3):363--390, 2016.

\bibitem{Wengs16sta}
Z.~Weng, J.Z. Yang, and X.~Lu.
\newblock A stabilized finite element method for the convection dominated
  diffusion optimal control problem.
\newblock {\em Appl. Anal.}, 95(12):2807--2823, 2016.

\bibitem{Xu82est}
J.~Xu.
\newblock Estimate of the convergence rate of the finite element solutions to
  elliptic equation of second order with discontinuous coefficients.
\newblock {\em Natural Science Journal of Xiangtan University}, 1:84--88, 1982.
\newblock (in Chinese).

\bibitem{Xus16opt}
J.~Xu and S.~Zhang.
\newblock Optimal finite element methods for interface problems.
\newblock In T.~Dickopf et~al., editor, {\em Domain Decomposition Methods in
  Science and Engineering XXII}, volume 104 of {\em Lecture Notes in
  Computational Science and Engineering}, pages 77--91. Springer, Cham, 2016.

\bibitem{Yangs17rob}
F.W. Yang, C.~Venkataraman, V.~Styles, and A.~Madzvamuse.
\newblock A robust and efficient adaptive multigrid solver for the optimal
  control of phase field formulations of geometric evolution laws.
\newblock {\em Commun. Comput. Phys.}, 21(1):65--92, 2017.

\bibitem{Zhang15imm}
Q.~Zhang, K.~Ito, Z.~Li, and Z.~Zhang.
\newblock Immersed finite elements for optimal control problems of elliptic
  pdes with interfaces.
\newblock {\em J. Comput. Phys.}, 298(C):305--319, 2015.

\bibitem{Zhus18spl}
H.~Zhu, K.~Liang, G.~He, and J.~Ying.
\newblock A splitting collocation method for elliptic interface problems.
\newblock {\em Numer. Math. Theor. Meth. Appl.}, 11:491--505, 2018.

\end{thebibliography}

\end{document}